\documentclass[10pt, a4paper]{amsart}
\usepackage[arXiv]{rseaineng}
\usepackage[cal]{rseaincmds}
\usepackage{hhline, array, bm}
\usepackage{ifthen}

\usepackage{cellspace}
\usepackage{tikz}
\tikzstyle{rt} = [circle, draw, thin, inner sep = 2.5pt, minimum size = 2pt, fill = white]
\tikzstyle{hrt} = [circle, draw, thin, inner sep = 2.5pt, minimum size = 2pt, fill = black]

\tikzstyle{v} = [circle, draw, inner sep=2pt, minimum size=3pt, fill = white]
\tikzstyle{vvv} = [circle, draw, ultra thick, inner sep=2pt, minimum size=2pt, fill = black]
\tikzstyle{v3} = [circle, draw, ultra thick, inner sep=2pt, minimum size=2pt, fill = black]
\tikzstyle{v2} = [circle, draw, ultra thick, inner sep=2pt, minimum size=2pt, fill = black]
\tikzstyle{v6} = [circle, draw, ultra thick, inner sep=2pt, minimum size=2pt, fill = black]
\tikzstyle{vv} = [circle, draw, thick, inner sep=5pt, minimum size=9pt, fill = white]
\tikzstyle{bv} = [circle, draw, ultra thick, inner sep=2pt, minimum size=2pt, fill = black]
\tikzstyle{rv} = [circle, draw, inner sep=1pt, minimum size=3pt, fill = red]

\allowdisplaybreaks

\newcommand{\varpivee}{\varpi^\vee}

\DeclareMathOperator{\stt}{R}
\newcommand{\st}{{\stt}}

\DeclareMathOperator{\sttr}{re}
\newcommand{\re}{{\sttr}}

\tikzstyle{wv} = [circle, draw, thin, inner sep=1.5pt, minimum size=1.5pt, fill = white]
\tikzstyle{bv} = [circle, draw, thin, inner sep=1.5pt, minimum size=1.5pt, fill = black]
\tikzstyle{vv} = [circle, draw, thin, inner sep=1.5pt, minimum size=1.5pt, fill = white]

\numberwithin{equation}{section}

\title[]{Equivariant version of the characteristic quasi-polynomials of root systems}
\author[Ryo Uchiumi]{Ryo Uchiumi}
\date{\today}
\thanks{The author was supported by JSPS KAKENHI, Grant Number 25KJ1735}

\begin{document}

\begin{abstract}
	An equivariant characteristic quasi-polynomial is a quasi-polynomial in $q$ consisting of the permutation character on the mod $q$ complement of the corresponding Coxeter arrangement.
	This concept is a refinement of the conventional characteristic quasi-polynomials of root systems.
	In this paper, we will show equivariant-theoretic refinements of the some properties of characteristic quasi-polynomials of root systems.
	Furthermore, we will explicitly compute equivariant characteristic quasi-polynomials of all irreducible reduced root systems and discuss the relationship with root systems constructed by the folding of the extended Dynkin diagrams.
\end{abstract}

\maketitle

\tableofcontents



\section{Introduction}

Let $L \cong \mathbb{Z}^\ell$ be a lattice and $\A$ be a hyperplane arrangement defined over $L$.
For each positive integer $q \in \mathbb{Z}_{>0}$, we consider the ``hyperplane arrangement"  $\A_q$ in $L/qL$ through the $q$-reduction.
Let $M(\A;q)$ be the complement of the arrangement $\A_q$ in $L/qL$.
Kamiya--Takemura--Terao \cite{KamiyaTakemuraTerao} introduced the \textbf{characteristic quasi-polynomial} $\chi_\A^\quasi$ of $\A$ as a counting function $\chi_{\A}^\quasi : q \lmapsto \#M(\A;q)$.
It is a quasi-polynomial in $q$ with gcd-property, that is, there exist a positive integer $\tilde{n}$ (\textbf{period}) and polynomials $f^{(1)},\ldots,f^{(\tilde{n})}$ (\textbf{constituents}) such that 
\begin{align}
	\chi_\A^\quasi(q) = f^{(r)}(q) \IF q \equiv r \pmod{\tilde{n}},
\end{align}
and $f^{(r_1)} = f^{(r_2)}$ if $\gcd\{\tilde{n},r_1\} = \gcd\{\tilde{n},r_2\}$.
In particular, for $r$ coprime to $\tilde{n}$, the constituent $f^{(r)}$ is equal to the \textbf{characteristic polynomial} $\chi_\A$ of $\A$, which is a polynomial defined by the combinatorial structure of $\A$ and the most important invariant, closely related to various aspects of $\A$.
Therefore the characteristic quasi-polynomial is a kind of refinement of the characteristic polynomial.
Furthermore, the notion of characteristic quasi-polynomials is, roughly speaking, the mod $q$ version of the Ehrhart theory.

Let $\PHI$ be a root system in a Euclidean space $E \cong \mathbb{R}^\ell$ and $\A_\PHI$ be the set of hyperplanes corresponding to each roots of $\PHI$ (\textbf{Coxeter arrangement}).
Then $\A_\PHI$ is a hyperplane arrangement defined over the coweight lattice $Z$ of $\PHI$.
The characteristic quasi-polynomial $\chi_\PHI^\quasi$ of $\PHI$ is defined as the characteristic quasi-polynomial of $\A_\PHI$ (with respect to $Z$).
The function $\chi_\PHI^\quasi$ has some interesting properties derived from the root system.
The first one is the following duality of $\chi_\PHI^\quasi$ (see \cref{thm2.2} \cref{thm2.2.3}):
\begin{align}
	\chi_\PHI^\quasi(q) = (-1)^\ell \chi_\PHI^\quasi(h-q), \label{dual}
\end{align}
where $h$ is the Coxeter number of $\PHI$.
The second one is that $\chi_\PHI^\quasi$ is almost equivalent to the Ehrhart quasi-polynomial $\Ell_{A_\circ}$ of the (open) fundamental alcove $A_\circ$ of $\PHI$ as follows (see \cref{yoshinagathm}):
\begin{align}
	\chi_\PHI^\quasi = \frac{\#W}{f}\cdot \Ell_{A_\circ}, \label{CEformula}
\end{align}
where $W$ is the Weyl group and $f$ is the index of connection of $\PHI$.

In \cite{Uchiumi}, the auther introduced an equivariant theory of the characteristic quasi-polynomials, inspired by the equivariant Ehrhart theory \cite{Stapledon}.
For an arrangement $\A$ invariant under a group action, the equivariant characteristic quasi-polynomial $\chi_{\A,q}$ of $\A$ is defined as the permutation character on the complement $M(\A;q)$, and it is a quasi-polynomial in $q$ (each constituent is a polynomial with class functions over $\mathbb{Q}$ as coefficients).
Substituting the identity element $1$ yields that $\chi_{\A,q}(1)$ is exactly equal to the characteristic quasi-polynomial $\chi_\A^\quasi(q)$.
Hence the equivariant characteristic quasi-polynomial is a kind of refinement of the characteristic quasi-polynomial.

In this paper, we consider the case of Coxeter arrangements.
For a root system $\PHI$, the arrangement $\A_\PHI$ is invariant under the action of the Weyl group $W$.
Thus the equivariant characteristic quasi-polynomial $\chi_{\PHI,q}$ of $\PHI$ is introduced as that of the Coxeter arrangement $\A_\PHI$ with respect to the action of $W$.
In \cite{Uchiumi}, the auther proved that $\chi_{\PHI,q}$ is exactly equal to the induced character of the equivariant Ehrhart quasi-polynomial $\chi_{A_\circ,q}$ of the (open) fundamental alcove $A_\circ$ (see \cref{thm4.2}):
\begin{align}
	\chi_{\PHI,q} = \Ind^W_{W_{A_\circ}}\chi_{A_\circ,q}, \label{indformula}
\end{align}
where $W_{A_\circ}$ is the subgroup of elements of $W$ that fix $A_\circ$ under the action of $w$ on $E/Z$.
It is an equivariant-theoretic refinement of the above formula \cref{CEformula}.
This paper also provides a refinement of the formula \cref{dual}.
The following is one of the main results of this paper:

\begin{theorem}[see \cref{thm4.6}]\label{thm1.1}
	Define a function $\delta : W \lra \mathbb{C}$ (character of $W$) by 
	\begin{align}
		\delta(w) \ceq (-1)^{\ell-r(w)},
	\end{align}
	where $r(w)$ is the dimention of the subspace of $E$ fixed by $w$.
	Then
	\begin{align}
		\chi_{\PHI,q} = (-1)^\ell \cdot \delta \cdot \chi_{\PHI,h-q}.
	\end{align}
\end{theorem}

By the formula \cref{indformula}, to compute $\chi_{\PHI,q}$, it is sufficient to compute $\chi_{\PHI,q}(w)$ only when $w \in W_{A_\circ}$.
Let $\widehat{W_\aff} \ceq Z \rtimes W$ be the \textbf{extended Weyl group} of $\PHI$, and $\widehat{\OMEGA}$ denote the stabilizer subgroup with respect to $A_\circ$.
We will prove that the group $W_{A_\circ}$ is equal to the image $\OMEGA = \pi(\widehat{\OMEGA})$ under the projection $\pi : \widehat{W_\aff} \lra W$ (see \cref{prop4.6}).
The structure of the group $\OMEGA$ has been clarified by \cite{Garnier, KomrakovPremet}, and it is known that $\OMEGA$ can be regarded as a subgroup of automorphisms of the extended Dynkin diagram of $\PHI$.

Furthermore, we will also discuss the constituents of the quasi-polynomial
$\chi_{\PHI,q}(w)$ for each $w \in W_{A_\circ}$.
Since $w \in W_{A_\circ} = \OMEGA$ is an automorphism of the extended Dynkin diagram, we can obtain a (finite) root system $\PHI^w_\re$ constructed by folding of the extended Dynkin diagram by $w$ (see \cite{UchiumiRoot}).
We will show that the constituents of $\chi_{\PHI,q}(w)$ can be obtained from the constituents of the characteristic quasi-polynomial of the hyperplane arrangement derived from another root system obtained by modifying $\PHI^w_\re$.
The following is one of the main results of this paper:

\begin{theorem}[see \cref{thm5.6}]
	Let $w \in W_{A_\circ}$ and $o(w)$ denote the order of $w$.
	Define a periodic function $c : \mathbb{Z}_{>0} \lra \mathbb{Z}$ by
	\begin{align}
		c(q) = \begin{cases*}
			0				& if $q \not\in o(w)\mathbb{Z}$;\\
			\varphi(o(w))	& if $q \in o(w)\mathbb{Z}$,
		\end{cases*}
	\end{align}
	where $\varphi$ is the Euler's totient function.
	Let $\PHI'$ be a root system (obtained by modifying $\PHI^w_\re$) and $d \in \mathbb{Z}_{>0}$ as given in \cref{tablePHId}.
	Then we have
	\begin{align}
		\chi_{\PHI,q}(w) = c(q) \cdot \chi_{d\PHI',\,M}^\quasi(q),
	\end{align}
	where $\chi_{d\PHI',\, M}^\quasi$ is a characteristic quasi-polynomial of the arrangement corresponding to the dilated root system $d\PHI'$ with respect to the coweight lattice $M$ of $\PHI_\re^w$ (see \cref{sec2.2.3} and \cref{sec2.3} for details).
\end{theorem}

As a corollary, we can obtain the following:
\begin{corollary}[see \cref{chichi}]
	If $\gcd\{f,q\} = 1$, then 
	\begin{align}
		\chi_{\PHI,q} = \frac{\chi_{\PHI}^\quasi(q)}{\#W} \cdot \chi_{\st},
	\end{align}
	where $\chi_{\st}$ is the regular character of $W$.
	Furthermore, let $\tilde{n}_\PHI$ denote the minimum period of $\chi_\PHI^\quasi$.
	If $\gcd\{f,\tilde{n}_\PHI,q\} = 1$, then 
	\begin{align}
		\chi_{\PHI,q} = \frac{\chi_{\A_\PHI}(q)}{\#W} \cdot \chi_{\st}.
	\end{align}
\end{corollary}

In a sense, the above corollary can be viewed as an equivaliant-refinement of the fact that $\chi_\PHI^\quasi(q) = \chi_{\A_\PHI}(q)$ if $q$ is coprime to $\tilde{n}_\PHI$.

The organization of this paper is as follows:
In \cref{sec2}, we will review some definitions and notations.
In particular, in \cref{sec2.3}, we consider a dilated root system $k\PHI$, dilate $\PHI$ by a factor of $k$, and compute its characteristic quasi-polynomial.
In \cref{sec2.4}, we summarize the results of the calculations of characteristic quasi-polynomials of root systems required for this paper.
In \cref{sec3.1}, we will recall the definition and notation of the equivariant version of the characteristic quasi-polynomials.
In \cref{sec3.2}, we confirm that the \cref{indformula} holds.
In \cref{sec3.3}, we will discuss the structure of the stabilizer group $W_{A_\circ}$.
In \cref{sec3.4}, we will prove the first main result \cref{thm4.6}.
The purpose of \cref{sec4} is to clarify the structure of the Weyl group action on the complement, in preparation for the second main result \cref{thm5.6}.
In \cref{secnew4.3}, we give a root system $\PHI_\re^w$ and express $\chi_{\PHI,q}(w)$ in terms of $\PHI_\re^w$.
The proof of \cref{chichi} is given at the end of \cref{secnew4.3}.
In \cref{sec5}, we will perform detailed calculations of $\chi_{\PHI,q}$ for each type of $\PHI$.
This completes the proof of \cref{thm5.6}.

\section{Characteristic quasi-polynomials of root systems} \label{sec2}

\subsection{Root system}\ 

\noindent
We recall some notions of root systems.
For more details, see \cite{Bourbaki, Humphreys}.

Let $E$ be a Euclidean space of dimension $\ell$ with inner product $(\cdot,\cdot)$, and $\PHI$ be an irreducible (not necessarily reduced) root system in $E$.
Note thet the emptyset $\PHI = \varnothing$ is a root system in the trivial Euclidean space $E = \{0\}$.
Let $Z$ denote the \textbf{coweigfht lattice} of $\PHI$ defined by
\begin{align}
	Z \ceq \bigset{x \in E}{(\alpha,x) \in \mathbb{Z} \tforall \alpha \in \PHI}.
\end{align}
Fix a basis $\DELTA = \{\alpha_1,\ldots,\alpha_\ell\}$ of $\PHI$ (set of simple roots), and let $\DELTA^\vee = \{\varpivee_1,\ldots,\varpivee_\ell\}$ be the dual basis of $\DELTA$, that is, 
\begin{align}
	(\alpha_i,\varpivee_j) = \delta_{ij}.
\end{align}
Then $Z$ is a free abelian group generated by $\DELTA^\vee$.
Therefore the root lattice 
\begin{align}
	Q \ceq \sum_{\alpha \in \PHI}\mathbb{Z}\alpha = \bigoplus_{\alpha_i \in \DELTA}\mathbb{Z}\alpha_i
\end{align}
can be regarded as the dual lattice of $Z$ by considering each root as a map via the inner product.
For each root $\alpha \in \PHI$, define $\alpha^\vee$, called a \textbf{coroot} of $\alpha$, as
\begin{align}
	\alpha^\vee \ceq \frac{2\alpha}{(\alpha,\alpha)},
\end{align}
and let $\veeQ$ be a lattice generated by coroots of $\PHI$, referred to as a \textbf{coroot lattice} of $\PHI$:
\begin{align}
	\veeQ = \sum_{\alpha \in \PHI}\mathbb{Z}\alpha^\vee.
\end{align}
It is well known that $\veeQ$ is a subgroup of $Z$ with finite index $f \ceq (Z:\veeQ)$ (called the \textbf{index of connection}).

Let $\tilde{\alpha}$ denote the highest root of $\PHI$.
Then it can be expressed as a linear combination
\begin{align}
	\tilde{\alpha} = n_1\alpha_1 + \cdots + n_\ell\alpha_\ell
\end{align}
for $n_1,\ldots,n_\ell \in \mathbb{Z}_{>0}$.
We define $\alpha_0 \ceq -\tilde{\alpha}$ and $n_0 = 1$.
It is known that $n_0 + \cdots + n_\ell$ is equal to the \textbf{Coxeter number} $h$ of $\PHI$.

Let $\PHI^+$ denote the set of positive roots corredponding to $\DELTA$.
For $\alpha \in \PHI^+$ and $k \in \mathbb{Z}$, define a hyperplane $H_{\alpha}^k$ by
\begin{align}
	H_\alpha^k \ceq \bigset{x \in E}{(\alpha,x) = k},
\end{align}
and let $s_{\alpha,k}$ denote the reflection with respect to $H_{\alpha}^k$.
It is known in \cite[\S4.9]{Humphreys} that the order of $W$ is equal to $f \cdot \ell! \cdot n_1 \cdots n_\ell$.
The \textbf{Weyl group} $W$ is the group generated by reflections $\tbigset{s_{\alpha,0}}{\alpha \in \PHI^+}$.
The \textbf{affine Weyl group} $W_\aff$ is the group generated by reflections $\tbigset{s_{\alpha,k}}{\alpha \in \PHI^+,\ k \in \mathbb{Z}}$.
It is well known that $W_\aff$ is the semidirect product $\veeQ \rtimes W$, where $\veeQ$ be regarded as the group of translations on $E$.
Let $M^\aff$ denote the complement of the hyperplanes $\A^\aff_\PHI = \tbigset{H_\alpha^k}{\alpha \in \PHI^+,\ k \in \mathbb{Z}}$.
Let $\mathcal{C}(\A^\aff)$ be the set of connected components of $M^\aff$.
Each eletemt of $\mathcal{C}(\A^\aff)$ is called an \textbf{alcove}.
In particular, the set 
\begin{align}
	A_\circ \ceq H_{\tilde{\alpha}}^{1,-} \cap \bigcap_{i = 1}^\ell H_{\alpha_i}^{0,+} = \bigset{x \in E}{(\alpha_i,x) > 0 \tforall i \in \{1,\ldots,\ell\},\ (\tilde{\alpha},x) < 1}
\end{align}
is an alcove, called the \textbf{fundamental alcove} of $\PHI$, where $H_{\tilde{\alpha}}^{1,-}$ and $H_{\alpha_i}^{0,+}$ denote the half spaces of $E$ defined as
\begin{align}
	H_{\tilde{\alpha}}^{1,-} = \bigset{x \in E}{(\tilde\alpha,x) < 1},\qquad
	H_{\alpha_i}^{0,+} = \bigset{x \in E}{(\alpha_i,x) > 0}.
\end{align}
The closure $\overline{A_\circ}$ is a convex hull
\begin{align}
	\overline{A_\circ} = \conv\Bigset{\frac{\varpivee_i}{n_i}}{i \in \{0,\ldots,\ell\}},
\end{align}
where $\varpivee_0 \ceq 0$.
Furthermore, $W_\aff$ acts simply transively on $\mathcal{C}(\A^\aff)$ \cite[\S4]{Humphreys}.
Thus the closure $\overline{A_\circ}$ is a fundamental domain for the action of $W$ on $E$.

For the classification and details of root systems, see \cref{roottable} and \cref{sec2.4}.

\begin{table}[h]
	
	\caption{Table of root systems}
	\label{roottable}
	
	\resizebox{\textwidth}{!}{
		\begin{tabular}{cc|c|c|c|c|c|c}
			\multicolumn{2}{c|}{$\PHI$}     & exponents                   & $n_1,\ldots,n_\ell$ & Coxeter number  $h$ & index of connection $f$ & order of the Weyl group            & minimum period $\tilde{n}_\PHI$ \\ \hline\hline
			$A_\ell$ & ($\ell \geq 1$) & $1,2,\ldots,\ell$           & $1,\ldots,1$                     & $\ell+1$            & $\ell+1$                & $(\ell+1)!$                        & $1$                             \\
			$B_\ell$ & ($\ell \geq 2$) & $1,3,\ldots,2\ell-1$        & $1,2,\ldots,2$                   & $2\ell$             & $2$                     & $2^\ell \cdot \ell!$               & $2$                             \\
			$C_\ell$ & ($\ell \geq 3$) & $1,3,\ldots,2\ell-1$        & $2,\ldots,2,1$                   & $2\ell$             & $2$                     & $2^\ell \cdot \ell!$               & $2$                             \\
			$D_\ell$ & ($\ell \geq 4$) & $1,3,\ldots,2\ell-3,\ell-1$ & $1,2\ldots,2,1,1$                & $2\ell-2$           & $4$                     & $2^{\ell-1} \cdot \ell!$           & $2$                             \\
			$E_6$    && $1,4,5,7,8,11$              & $1,2,2,3,2,1$                    & $12$                & $3$                     & $2^7\cdot 3^4 \cdot 5$             & $6$                             \\
			$E_7$    && $1,5,7,9,11,13,17$          & $2,2,3,4,3,2,1$                  & $18$                & $2$                     & $2^10 \cdot 3^4 \cdot 5 \cdot 7$   & $12$                            \\
			$E_8$    && $1,7,11,13,17,19,23,29$     & $2,3,4,6,5,4,3,2$                & $30$                & $1$                     & $2^14 \cdot 3^5 \cdot 5^2 \cdot 7$ & $60$                            \\
			$F_4$    && $1,5,7,11$                  & $2,3,4,2$                        & $12$                & $1$                     & $2^7 \cdot 3^2$                    & $12$                            \\
			$G_2$    && $1,5$                       & $1,2$                            & $6$                 & $1$                     & $2^2 \cdot 3$                      & $6$                            
		\end{tabular}
	}
	
	\

\end{table}

\subsection{Characteristic quasi-polynomial}
\subsubsection{For general hyperplane arrangements}\ 

\noindent
Let $L \cong \mathbb{Z}^\ell$ be a lattice, and $L^\vee \ceq \Hom_\mathbb{Z}(L,\mathbb{Z})$ be the dual lattice of $L$.
We set $L_\mathbb{R} \ceq L \otimes \mathbb{R}$.
For $\beta_1,\dots,\beta_n \in L^\vee$ and $k_1,\ldots,k_n \in \mathbb{Z}$, define a hyperplane arrangement $\A = \{H_1,\ldots,H_n\}$ by
\begin{align}
	H_i \ceq H_{\beta_i}^{k_i} = \bigset{x \in L_\mathbb{R}}{\beta_i(x) = k_i}.
\end{align}
For a positive integer $q \in \mathbb{Z}_{>0}$, let $\pi_q :L \lra L/qL$ denote the natural projection.
Give a ``hyperplane" 
\begin{align}
	H_{i,q} \ceq \bigset{\pi_q(x) \in L/qL}{\beta_i(x) \not\equiv k_i \pmod{q}}
\end{align}
in $L/qL$, and ``hyperplane arrangement" $\A_q = \{H_{1,q},\ldots,H_{n,q}\}$ in $L/qL$.
Let $M(\A;q)$ be the complement of $\A_q$ in $L/qL$, that is, 
\begin{align}
	M(\A;q) &\ceq (L/qL) \setminus \bigcup_{i=1}^nH_{i,q}\\
	&= \bigset{\pi_q(x) \in L/qL}{\beta_i(x) \not\equiv k_i \pmod{q} \tforall i \in \{1,\ldots,n\}}.
\end{align}

\begin{theorem}[{\cite[Theorem 2.4]{KamiyaTakemuraTerao}}, {\cite[Theorem 3.1]{KamiyaTakemuraTeraoNon}}]\label{KamiyaTakemuraTerao}
	There exists a non-negative integer $q_0 \in \mathbb{Z}_{\geq 0}$ such that $\#M(\A;q)$ is a quasi-polynomial in $q > q_0$, that is, there exist a positive integer $\tilde{n}$ and polynomials $f^{(1)},\ldots,f^{(\tilde{n})}$ such that 
	\begin{align}
		\#M(\A;q) = f^{(r)}(q) \IF r \equiv q \pmod{\tilde{n}}
	\end{align}
	for all $q > q_0$.
	Furthermore, it satisfies the following:
	\begin{iitemize}
		\item $f^{(r)}$ is monic and $\deg{f^{(r)}} = \ell$ for all $r$.
		\item $f^{(r_1)} = f^{(r_2)}$ if $\gcd\{\tilde{n},r_1\} = \gcd\{\tilde{n},r_2\}$ (\textbf{gcd-property}).
		\item $f^{(1)}$ is equal to the characteristic polynomial of $\A$ (see also \cite{Athanasiadis1996}).
		\item If $k_1 = \cdots = k_n = 0$ (i.e., $\A$ is a central arrangement), then $q_0 = 0$.
	\end{iitemize}
\end{theorem}
The function $\chi_\A^\quasi : q \lmapsto \#M(\A;q)$ is called the \textbf{chracteristic quasi-polynomial} of $\A$.
The number $\tilde{n}$ is called a \textbf{period} of $\chi_\A^\quasi$, and $f^{(r)}$ is called an \textbf{$\bm r$-th constituent} of $\chi_\A^\quasi$.

Let $\{\alpha_1,\ldots,\alpha_\ell\}$ be a $\mathbb{Z}$-basis of $L^\vee$.
Define an $\ell \times n$ matrix $S = (s_{ij})_{i,j}$ such that
\begin{align}
	\beta_j = s_{1j}\alpha_1 + \cdots + s_{\ell j}\alpha_\ell
\end{align}
for all $j \in \{1,\ldots,n\}$.
Let $K \ceq (k_1,\ldots,k_n)$.
Then we have
\begin{align}
	\chi_\A^\quasi(q) = \#\bigset{z \in (\mathbb{Z}/q\mathbb{Z})^\ell}{\text{$zS + K$ has no zero components}}.
\end{align}
In the case where $k_1 = \cdots = k_n = 0$, the characteristic quasi-polynomial $\chi_\A^\quasi$ can be calculated using the elementary divisors theory as follows:
For a nonempty subset $J \subseteq\{1,\ldots,n\}$, let $S_J$ be the $\ell \times \#J$ submatrix of $S$ corresponding to $J$, $r(J) \ceq \rank{S_{J}}$, and $d_{J,1},\ldots,d_{J,r(J)}$ be the elementary divisors of $S_{J}$.
When $J = \emptyset$, we agree that $r(J) = 0$.
Then we have the following \cite[Equation (7), (10)]{KamiyaTakemuraTerao}:
\begin{align}
	\chi_\A^\quasi(q) = \sum_{J \subseteq \{1,\ldots,n\}}(-1)^{\#J} \left( \prod_{j=1}^{r(J)}\gcd\{d_{J,j},q\} \right) q^{\ell-r(J)}.
\end{align}
Furthermore, 
\begin{align}
	\tilde{n}_\A \ceq \lcm\bigset{d_{J,j}}{J \subseteq \{1,\ldots,n\},\ J \neq \emptyset,\ j \in \{1,\ldots,r(J)\}}
\end{align}
is the minimum period of $\chi_\A^\quasi$ \cite[Theorem 1.2]{HigashitaniTranYoshinaga}.

The two constituents of the characteristic quasi-polynomial have the same coefficients for higher-degree terms.
For $s \in \{1,\ldots,\ell\}$, define
\begin{align}
	\mathcal{E}_s \ceq \bigset{d_{J,r(J)}}{J \subseteq \{1,\ldots,n\},\ 1 \leq \#J \leq s}.
\end{align}

\begin{proposition}[{\cite[Corollary 2.3]{KamiyaTakemuraTeraoRoot}}]\label{consticoef}
	Let $f^{(r_1)}$ and $f^{(r_2)}$ be the $r_1$-th and $r_2$-th constituents of $\chi_\A^\quasi$.
	For $s \in \{1,\ldots,\ell\}$, if $\gcd\{e,r_1\} = \gcd\{e,r_2\}$ for all $e \in \mathcal{E}_s$, then 
	\begin{align}
		\deg(f^{(r_1)} - f^{(r_2)}) < \ell-s,
	\end{align}
	that is, $f^{(r_1)}$ and $f^{(r_2)}$ have the same coefficients of degree $\ell-s$ and higher.
\end{proposition}

\subsubsection{The coset method}\ 

\noindent
We will introduce a method for calculating $\chi_\A^\quasi$ by decomposing into several sums.
Let $M$ be a sublattice of $L$, and suppose that $\rank{M} = \rank{L} = \ell$.
Let $\{\lambda_1,\ldots,\lambda_\ell\}$ be a $\mathbb{Z}$-basis for $L$ and $\{\mu_1,\ldots,\mu_\ell\}$ be a $\mathbb{Z}$-basis for $M$.
Define $P = (p_{ij})_{ij}$ be an $\ell \times \ell$ matrix satisfies 
\begin{align}
	\mu_i = p_{i1}\lambda_1 + \cdots + p_{i\ell}\lambda_\ell 
\end{align}
for all $i \in \{1,\ldots,\ell\}$.
Let $S_{\!M} \ceq PS$.
For $q \in \mathbb{Z}_{>0}$, let $\mathbb{Z}_q \ceq \mathbb{Z}/q\mathbb{Z}$.
Define a homomorphism $\pi : \mathbb{Z}_q^\ell \lra \mathbb{Z}_q^\ell$ by $\pi(z) = zP$.
Then $\im\pi$ is a subgroup of $\mathbb{Z}_q^\ell$ with finite index $b(q) \ceq (\mathbb{Z}_q^\ell : \im\pi)$.
Let $\{g_1,\ldots,g_{b(q)}\}$ denote the complete set of coset representatives of $\mathbb{Z}_q^\ell/\im\pi$.
For $i \in \{1,\ldots,b(q)\}$, define 
\begin{align}
	f_i(q) \ceq \#\bigset{z \in \mathbb{Z}_q^\ell}{\text{$zS_{\!M} + g_iS$ has no zero columns}}.
\end{align}
Note that there exists $i_0 \in \{1,\ldots,b(q)\}$ uniquely such that $g_{i_0} \in \im\pi$, and $f_{i_0}(q)$ is equal to the characteristic quasi-polynomial of $\A$ with respect to the lattice $M$.

\begin{theorem}[Coset method {\cite[Theorem 4.1]{KamiyaTakemuraTeraoRoot}}]\label{cosetmethod}
	Under the above settings, the following holds:
	\begin{align}
		\chi_{\A}^\quasi(q) = \frac{1}{b(q)}\sum_{i=1}^{b(q)}f_i(q).
	\end{align}
\end{theorem}

\subsubsection{For root systems}\ \label{sec2.2.3}

\noindent
Let $\PHI$ be an irreducible root system, and let $L$ be a sublattice of $Z$ satisfying $\rank{L} = \ell$.
For $q \in \mathbb{Z}_{>0}$, define a set
\begin{align}
	M(\PHI,L;q) \ceq \bigset{\pi_q(x) \in L/qL}{(\beta,x) \not\equiv 0 \pmod{q} \tforall \beta \in \PHI^+}
\end{align}
and a function $\chi_{\PHI,L}^\quasi : q \lmapsto \#M(\PHI,L;q)$.
Then $\chi_{\PHI,L}^\quasi$ is equal to the characteristic quasi-polynomial of $\A_\PHI = \tbigset{H_\beta^0}{\beta \in \PHI^+}$ with respect to $L$.
In the case where $L = Z$, the function $\chi_\PHI^\quasi \ceq \chi_{\PHI,Z}^\quasi$ is called the \textbf{characteiristic quasi-polynomial} of $\PHI$.
If $\PHI = \varnothing$, then we have $\chi_{\varnothing}^\quasi(q) = 1$ for all $q \in \mathbb{Z}_{>0}$.
Let $\tilde{n}_\PHI$ denote the minimum period of $\chi_\PHI^\quasi$.
See \cref{roottable} for the value of $\tilde{n}_\PHI$ for each root system.

The characteristic quasi-polynomial $\chi_\PHI^\quasi$ has the following special properties.
\begin{theorem}\label{thm2.2}
	Let $\PHI$ be an irreducible reduced root system with the exponents $e_1,\ldots,e_\ell$ and the Coxeter number $h$.
	\begin{eenumerate}
		\item The 1st constituent of $\chi_\PHI^\quasi$ (i.e., the characteristic polynomial $\chi_{\A_\PHI}$ of $\A_\PHI$) is equal to the following (\cite[Corollary 3.3]{OrlikSolomonTerao}): 
		\begin{align}
			\chi_{\A_\PHI}(t) = (t - e _1) \cdots (t - e_\ell).
		\end{align}
		\item Let $q \in \mathbb{Z}_{>0}$.
		Then $\chi_\PHI^\quasi(q) > 0$ if and only if $q \geq h$ (\cite[Theorem 7.2]{KamiyaTakemuraTeraoRoot}).
		\item\label{thm2.2.3} $\chi_\PHI^\quasi(q) = (-1)^\ell \chi_\PHI^\quasi(h-q)$, where 
		\begin{align}
			\chi_\PHI^\quasi(q) = f^{(r)}(q) \IF q \equiv r \pmod{\tilde{n}_\PHI}
		\end{align}
		for $q < 0$ (\cite[Corollary 3.8]{Yoshinaga}).
	\end{eenumerate}
\end{theorem}

For $q \in \mathbb{Z}_{>0}$, define
\begin{align}
	\Ell_{A_\circ}(q) \ceq \#(qA_\circ \cap Z),\qquad
	\Ell_{\overline{A_\circ}}(q) \ceq \#(q\overline{A_\circ} \cap Z).
\end{align}
These are quasi-polynomial in $q$, and called \textbf{Ehrhart quasi-polynomials} of $A_\circ$ and $\overline{A_\circ}$.
It is well known that the following \textbf{Ehrhart reciprocity} holds:
\begin{align}
	\Ell_{A_\circ}(q) = (-1)^{\ell}\Ell_{\overline{A_\circ}}(-q).
\end{align}

\begin{theorem}[{\cite[Proposition 3.7]{Yoshinaga}}, see also \cite{Suter}]\label{yoshinagathm}
	The characteristic quasi-polynomial $\chi_\PHI^\quasi$ is equal to an integer multiple of the Ehrhart quasi-polynomial of $A_\circ$:
	\begin{align}
		\chi_\PHI^\quasi(q) = \frac{\#W}{f} \cdot \Ell_{A_\circ}(q) = \frac{\#W}{f} \cdot (-1)^\ell \cdot \Ell_{\overline{A_\circ}}(-q). \label{yoshinagaeq}
	\end{align}
\end{theorem}

\subsection{Characteristic quasi-polynomial of dilated root systems}\ \label{sec2.3}

\noindent
Let $\PHI$ be a root system and $Z$ denote its coweight lattice.
For $k \in \mathbb{Z}_{>0}$, define a set $k\PHI$ by 
\begin{align}
	k\PHI \ceq \bigset{k\beta}{\beta \in \PHI}.
\end{align}
It is clear that $k\PHI$ is a root system isomorphic to $\PHI$.
The characteristic quasi-polynomial $\chi_{k\PHI,Z}^\quasi$ of the dilated root system $k\PHI$ can be computed as follows.

\begin{theorem}\label{dilated}
	Let $\PHI$ be a root system with the coweight lattice $Z$, and $k \in \mathbb{Z}_{>0}$.
	For $r \in \mathbb{Z}_{>0}$ and $g \ceq \gcd\{k,r\}$, the $r$-th constituent $f_k^{(r)}$ of $\chi_{k\PHI,Z}^\quasi$ can be expressed using the $(g^{-1}r)$-th constituent $f_1^{(g^{-1}r)}$ of $\chi_\PHI^\quasi$ as follows: 
	\begin{align}
		f_k^{(r)}(t) = g^\ell \cdot f_1^{(g^{-1}r)}(g^{-1}t).
	\end{align}
	Therefore the quasi-polynomial $\chi_{k\PHI,Z}^\quasi$ has the minimum period $k\tilde{n}_\PHI$.
\end{theorem}
\begin{proof}
	Let $S$ be the coefficient matrix of $\PHI^+$ with respect to $Z$.
	Then
	\begin{align}
		\chi_{k\PHI,Z}^\quasi(q) = \#\bigset{z \in (\mathbb{Z}/q\mathbb{Z})^\ell}{\text{$z(kS)$ has no zero components}}
	\end{align}
	for each $k \in \mathbb{Z}_{>0}$.
	For any integer matrix $X$ and $k \in \mathbb{Z}_{>0}$, the elementary divisors of the matrix $kX$ is exactly $k$ times the elementary divisors of $X$.
	Hence we can see that 
	\begin{align}
		f_k^{(r)}(t) = \sum_{J \subseteq \{1,\ldots,n\}}(-1)^{\#J} \left( \prod_{j=1}^{r(J)}\gcd\{kd_{J,j},r\} \right) t^{\ell-r(J)},
	\end{align}
	where $S_J$ is the submatrix corresponding to $J$, $r(J) = \rank{S_J}$, and $\{d_{J,1},\ldots,d_{J,r(J)}\}$ is the elementary divisors of $S_J$.
	Since $c \cdot \gcd\{a,b\} = \gcd\{ca,cb\}$ for all $a,b,c \in \mathbb{Z}_{>0}$, we have
	\begin{align}
		f_k^{(r)}(t) &= \sum_{J \subseteq \{1,\ldots,n\}}(-1)^{\#J} \left( \prod_{j=1}^{r(J)}\gcd\{kd_{J,j},r\} \right) t^{\ell-r(J)}\\
		&= \sum_{J \subseteq \{1,\ldots,n\}}(-1)^{\#J} \left( \prod_{j=1}^{r(J)}\Bigl( g \cdot \gcd\{g^{-1}kd_{J,j},\, g^{-1}r\} \Bigr) \right) t^{\ell-r(J)}\\
		&= \sum_{J \subseteq \{1,\ldots,n\}}(-1)^{\#J} \left( \prod_{j=1}^{r(J)}\gcd\{d_{J,j},\, g^{-1}r\} \right) g^{r(J)}t^{\ell-r(J)}\\
		&= g^\ell\sum_{J \subseteq \{1,\ldots,n\}}(-1)^{\#J} \left( \prod_{j=1}^{r(J)}\gcd\{d_{J,j},\, g^{-1}r\} \right) (g^{-1}t)^{\ell-r(J)}\\
		&= g^\ell \cdot f_1^{(g^{-1}r)}(g^{-1}t). \qedend
	\end{align}
\end{proof}

For a specific example, see \cref{exampleofdilated}.

\subsection{List of characteristic quasi-polynomials of root systems}\ \label{sec2.4}

\noindent
This subsection summarizes the results of the calculations for characteristic quasi-polynomial used in this paper.
Most of these were computed in \cite{KamiyaTakemuraTeraoRoot}.

\subsubsection{type $A_\ell$}\ 

\noindent
Let $\{e_1,\ldots,e_{\ell+1}\}$ be a standard basis for $\mathbb{R}^{\ell+1}$, and define
\begin{align}
	E \ceq \bigset{(x_1,\ldots,x_{\ell+1}) \in \mathbb{R}^{\ell+1}}{x_1 + \cdots + x_{\ell+1} = 0}.
\end{align}
Then
\begin{align}
	\PHI = \bigset{\pm(e_i - e_j)}{i,j \in \{1,\ldots,\ell+1\},\ i < j}
\end{align}
is a root system of type $A_\ell$.
The roots $\DELTA = \{e_1-e_2, \, \ldots, \, e_\ell-e_{\ell+1}\}$ is a basis of $\PHI$.
Then
\begin{align}
	\DELTA^\vee = \Bigset{(e_1 + \cdots + e_i) - \frac{i}{\ell+1}(e_1 + \cdots + e_{\ell+1})}{i \in \{1,\ldots,\ell\}}.
\end{align}

\begin{proposition}[type $A_\ell$, {\cite[\S3]{KamiyaTakemuraTeraoRoot}}]\label{typeAroot}
	Let $\PHI$ be a root system of type $A_\ell$.
	Then
	\begin{align}
		\chi_\PHI^\quasi(q) = (q-1) \cdots (q-\ell).
	\end{align}
\end{proposition}

\subsubsection{type $B_\ell$}\ 

\noindent
Let $\{e_1,\ldots,e_\ell\}$ be a standard basis for $E = \mathbb{R}^\ell$.
Then
\begin{align}
	\PHI = \bigset{\pm(e_i-e_j),\ \pm(e_i+e_j),\ \pm e_k}{i,j,k \in \{1,\ldots,\ell\},\ i < j}
\end{align}
is a root system of type $B_\ell$.
The roots $\DELTA = \{e_1-e_2, \, \ldots, \, e_{\ell-1}-e_\ell, \, e_\ell\}$ is a basis of $\PHI$.
Then
\begin{align}
	\DELTA^\vee = \bigset{e_1 + \cdots + e_i}{i \in \{1,\ldots,\ell\}}.
\end{align}
Hence the coweight lattice $Z$ is isomorphic to the lattice generated by $\{e_1,\ldots,e_\ell\}$.

\begin{proposition}[type $B_\ell$, {\cite[Theorem 4.8]{KamiyaTakemuraTeraoRoot}}]\label{typeBroot}
	Let $\PHI$ be a root system of type $B_\ell$.
	Then
	\begin{align}
		\chi_\PHI^\quasi(q) = \begin{cases*}
			(q-1)(q-3) \cdots (q-(2\ell-1)) 		& \textup{if $q \not\in 2\mathbb{Z}$};\\
			(q-2)(q-4) \cdots (q-2(\ell-1))(q-\ell) & \textup{if $q \in 2\mathbb{Z}$}.
		\end{cases*}
	\end{align}
\end{proposition}

\subsubsection{type $C_\ell$}\ 

\noindent
Let $\{e_1,\ldots,e_\ell\}$ be a standard basis for $E = \mathbb{R}^\ell$.
Then
\begin{align}
	\PHI = \bigset{\pm(e_i-e_j),\ \pm(e_i+e_j),\ \pm 2e_k}{i,j,k \in \{1,\ldots,\ell\},\ i < j}
\end{align}
is a root system of type $B_\ell$.
The roots $\DELTA = \{e_1-e_2, \, \ldots, \, e_{\ell-1}-e_\ell, \, 2e_\ell\}$ is a basis of $\PHI$.
Then
\begin{align}
	\DELTA^\vee = \bigset{e_1 + \cdots + e_i}{i \in \{1,\ldots,\ell-1\}} \sqcup \left\{ \frac{e_1 + \cdots + e_\ell}{2} \right\}.
\end{align}
Hence the lattice generated by $\{e_1,\ldots,e_\ell\}$ is a subgroup of the coweight lattice $Z$ of index $2$.

\begin{proposition}[type $C_\ell$, {\cite[Theorem 4.8]{KamiyaTakemuraTeraoRoot}}]\label{typeCroot}
	Let $\PHI$ be a root system of type $C_\ell$.
	Then
	\begin{align}
		\chi_\PHI^\quasi(q) = \begin{cases*}
			(q-1)(q-3) \cdots (q-(2\ell-1)) 		& \textup{if $q \not\in 2\mathbb{Z}$;}\\
			(q-2)(q-4) \cdots (q-2(\ell-1))(q-\ell) & \textup{if $q \in 2\mathbb{Z}$}.
		\end{cases*}
	\end{align}
	Let $L$ be a lattice generated by $\{e_1,\ldots,e_\ell\}$.
	Then
	\begin{align}
		\chi_{\PHI,L}^\quasi(q) = \begin{cases*}
			(q-1)(q-3) \cdots (q-(2\ell-1)) 		& \textup{if $q \not\in 2\mathbb{Z}$;}\\
			(q-2)(q-4) \cdots (q-2(\ell-1))(q-2\ell) & \textup{if $q \in 2\mathbb{Z}$}.
		\end{cases*}
	\end{align}
\end{proposition}
\begin{proof}
	The function $\chi_{\PHI,L}^\quasi$ is exactly equal to $\chi_{T(C_\ell)}$ in \cite[Proposition 4.6, Theorem 4.7]{KamiyaTakemuraTeraoRoot}
\end{proof}

Let $\PHI(B_\ell)$ be a root system of type $B_\ell$.
The dilated root system $2\PHI(B_\ell)$ can be regarded as the subset of $Z^\vee$.
Thus we can consider the characteristic quasi-polynomial $\chi_{2\PHI(B_\ell),\,Z}^\quasi$.

\begin{proposition}\label{prop2.9}
	Let $\PHI(B_\ell)$ be a root system of type $B_\ell$, and let $Z$ be the coweight lattice of a root system of type $C_\ell$.
	Then
	\begin{align}
		\chi_{2\PHI(B_\ell),\,Z}^\quasi(q) = 
		\begin{cases*}
			(q-1)(q-3) \cdots (q-(2\ell-1))					& \textup{if $\gcd\{4,q\} = 1$;}\\
			(q-2)(q-6) \cdots (q-2(2\ell-3))(q-(3\ell-2))	& \textup{if $\gcd\{4,q\} = 2$;}\\
			(q-4)(q-8) \cdots (q-4(\ell-1))(q-\ell)			& \textup{if $\gcd\{4,q\} = 4$.}
		\end{cases*}
	\end{align}
\end{proposition}
\begin{proof}
	Let $L$ be the lattice generated by $\{e_1,\ldots,e_\ell\}$.
	We compute $\chi_{2\PHI(B_\ell),\,Z}^\quasi$ by applying the coset method to $\chi_{2\PHI(B_\ell),\,L}^\quasi$.
	Since $L$ is isomorphic to the coweight lattice of a root system of type $B_\ell$, \cref{dilated} implies that
	\begin{align}
		\chi_{2\PHI(B_\ell),\, L}^\quasi(q) = \begin{cases*}
			(q-1)(q-3) \cdots (q-(2\ell-1)) 			& \textup{if $\gcd\{4,q\} = 1$;}\\
			(q-2)(q-6) \cdots (q-2(2\ell-1)) 			& \textup{if $\gcd\{4,q\} = 2$;}\\
			(q-4)(q-8) \cdots (q-4(\ell-1))(q-2\ell) 	& \textup{if $\gcd\{4,q\} = 4$.}
		\end{cases*} \label{exampleofdilated}
	\end{align}
	Let $S_Z$ and $S_L$ denote the coefficient matrices of $2\PHI(B_\ell)^+$ with respect to $Z$ and $L$, respectively.
	For example, when $\ell = 2$, we have 
	\begin{align}
		2\PHI(B_2)^+ = \{2\alpha_1,\ \alpha_2,\ 2\alpha_1 + \alpha_2,\ 2\alpha_1 + 2\alpha_2\} = \{2e_1 - 2e_2,\ 2e_2,\ 2e_1,\ 2e_1 + 2e_2\},
	\end{align}
	and 
	\begin{align}
		S_Z = \begin{pmatrix}
			2 & 0 & 2 & 2\\
			0 & 1 & 1 & 2
		\end{pmatrix},\qquad
		S_L = \begin{pmatrix}
			2	& 0 & 2 & 2\\
			-2	& 2 & 0 & 2
		\end{pmatrix}.
	\end{align}
	The transformation matrix $P$ is given as
	\begin{align}
		P = \left(\begin{array}{rrrrrr}
			1 \\
			-1 & 1 \\
			& -1 & \ddots\\
			& & \ddots & 1\\
			&&& -1 & 1\\
			&&&& -1 & 2
		\end{array}\right),
	\end{align}
	that is, $S_L = PS_Z$.
	Let $q \in \mathbb{Z}_{>0}$.
	For a homomorphism $\pi : \mathbb{Z}_q^\ell \lra \mathbb{Z}_q^\ell$ defined by $z \lmapsto zP$, we have $b(q) = (\mathbb{Z}_q^\ell:\im\pi) = \gcd\{2,q\}$.
	Hence, if $q$ is odd, then 
	\begin{align}
		\chi_{2\PHI(B_\ell),\,Z}^\quasi(q) = \chi_{2\PHI(B_\ell),\, L}^\quasi(q) = (q-1)(q-3) \cdots (q-(2\ell-1)).
	\end{align}
	
	Suppose that $q$ is even.
	Let $g_1 \ceq (0,\ldots,0)$, $g_2 \ceq (0,\ldots,0,1) \in \mathbb{Z}_q^\ell$.
	Then $\{g_1,g_2\}$ is a complete set of coset representatives of $\mathbb{Z}_q^\ell/\im\pi$.
	For $i \in \{1,2\}$, let
	\begin{align}
		f_i(q) \ceq \#\bigset{z \in \mathbb{Z}_q^\ell}{\text{$zS_{\!L} + g_iS_{\!Z}$ has no zero components}}.
	\end{align}
	Then $f_1(q) = \chi_{2\PHI(B_\ell),\, L}^\quasi(q)$, and
	\begin{align}
		f_2(q) 
		&= \#\bigset{(z_1,\ldots,z_\ell) \in \mathbb{Z}_q^\ell}{2z_i - 2z_j \not\equiv 0,\ \ 2z_i + 2z_j + 2 \not\equiv 0,\ \ 2z_i + 1 \not\equiv 0 \pmod{q}}\\
		&= \#\Bigset{(z_1,\ldots,z_\ell) \in \mathbb{Z}_q^\ell}{\begin{rgathered}
				z_i - z_j \not\equiv 0,\ \ z_i - z_j \not\equiv \tfrac{q}{2},\ \ z_i + z_j + 1 \not\equiv 0,\ \ z_i + z_j + 1 \not\equiv \tfrac{q}{2} \pmod{q}\\ \tforall i,j \in \{1,\ldots,\ell\},\ i \neq j
			\end{rgathered}
			}.
	\end{align}
	For $m \in \mathbb{Z}_{>0}$, define a set
	\begin{align}
		X_q^m \ceq \Bigset{(z_1,\ldots,z_m) \in \mathbb{Z}_q^m}{\begin{rgathered}
				z_i - z_j \not\equiv 0,\ \ z_i - z_j \not\equiv \tfrac{q}{2},\ \ z_i + z_j + 1 \not\equiv 0,\ \ z_i + z_j + 1 \not\equiv \tfrac{q}{2} \pmod{q}\\ \tforall i,j \in \{1,\ldots,m\},\ i \neq j
		\end{rgathered}}.
	\end{align}
	For $a \in \mathbb{Z}_q$, let $X_q^m(a) \ceq \bigset{(z_1,\ldots,z_m) \in X_q^m}{z_1 = a}$.
	Then, for $(z_1,\ldots,z_m) \in X_q^m(a)$ and $i \in \{2,\ldots,m\}$, the value $z_i$ satisfies
	\begin{align}
		z_i \not\equiv a,\quad
		z_i \not\equiv \frac{q}{2} - 1 - a,\quad 
		z_i \not\equiv \frac{q}{2} + a,\quad 
		z_i \not\equiv q - 1 - a \pmod{q}.
	\end{align}
	Note that, if $a \equiv \frac{q-2}{4}$ or $a \equiv \frac{3q - 2}{4}$, then
	\begin{align}
		a \equiv \frac{q}{2} - 1 - a \AND \frac{q}{2} + a \equiv q-1-a.
	\end{align}
	Hence we can see that 
	\begin{align}
		\#X_q^m &= \sum_{a \in \mathbb{Z}_q^m} \#X_q^m(a)\\
		&= \sum_{a \in \mathbb{Z}_q^m} \#\Bigset{(z_1,\ldots,z_{m-1}) \in \mathbb{Z}_q^{m-1}}{
			\begin{multlined}[b]
				z_i \not\equiv a,\ \ z_i \not\equiv \tfrac{q}{2} - 1 - a,\ \ z_i \not\equiv \tfrac{q}{2} + a,\ \ z_i \not\equiv q-1-a\\
				z_i - z_j \not\equiv 0,\ \ z_i - z_j \not\equiv \tfrac{q}{2},\ \ z_i + z_j + 1 \not\equiv 0,\ \ z_i + z_j + 1 \not\equiv \tfrac{q}{2} \pmod{q}\\
				\tforall i,j \in \{1,\ldots,m-1\},\ i \neq j\vspace{-6.5mm}
			\end{multlined}
		}\\
		&= \begin{cases*}
			(q-2) \cdot \#X_{q-4}^{m-1} + 2 \cdot \#X_{q-2}^{m-1} 	& if $\gcd\{4.q\} = 2$;\\
			q \cdot \#X_{q-4}^{m-1} 								& if $\gcd\{4,q\} = 4$
		\end{cases*}
	\end{align}
	for all $m,q \in \mathbb{Z}_{>0}$.
	
	\begin{lemma}\label{lem4.22}
		For $m,q \in \mathbb{Z}_{>0}$, we have
		\begin{align}
			\#X_q^m = \begin{cases*}
				(q-2)(q-6) \cdots (q-2(2m-3))(q-2(m-1)) 	& \textup{if $\gcd\{4.q\} = 2$;}\\
				q(q-4) \cdots (q-4(m-1))				& \textup{if $\gcd\{4,q\} = 4$.}
			\end{cases*}
		\end{align}
	\end{lemma}
	\begin{proof}[Proof of \cref{lem4.22}]
		It is clear that $\#X_q^1 = q$.
		Let $m > 1$.
		Suppose that $\gcd\{4,q\} = 4$.
		By induction on $m$, we have
		\begin{align}
			\#X_q^m = q \cdot \#X_{q-4}^{m-1} = q(q-4) \cdots (q-4(m-1)).
		\end{align}
		Suppose that $\gcd\{4,q\} = 2$.
		By definition, we have $\#X_2^m = 0$.
		Therefore we can assume that $q \geq 6$.
		Then $\#X_{q-2}^{m-1} = (q-2) \cdot \#X_{q-6}^{m-2}$ since $\gcd\{4,q-2\} = 4$.
		Thus, we have
		\begin{align}
			\#X_q^m &= (q-2) \cdot \#X_{q-4}^{m-1} + 2(q-2) \cdot \#X_{q-6}^{m-2}\\
			&= (q-2)\Bigl( (q-6) \cdot \#X_{q-8}^{m-2} + 2 \cdot \#X_{q-6}^{m-2} \Bigr) + 2(q-2) \cdot \#X_{q-6}^{m-2}\\
			&= (q-2)\Bigl((q-6) \cdot \#X_{q-8}^{m-2} + 4 \cdot \#X_{q-6}^{m-2}\Bigr).
		\end{align}
		For $k \in \{1,\ldots,m-1\}$, define
		\begin{align}
			Y(k,q) \ceq (q-2) \cdot \#X_{q-4}^{m-k} + 2k \cdot \#X_{q-2}^{m-k}.
		\end{align}
		They satisfy
		\begin{align}
			Y(k,q) = \begin{cases*}
				(q-2) \cdot Y(k+1,\,q-4) 	& if $k < m-1$;\\
				(q-2)(q+2m-6)				& if $k = m-1$.
			\end{cases*}
		\end{align}
		Hence we have 
		\begin{align}
			\#X_q^m = Y(1,q) = (q-2)(q-6) \cdots (q-2(2m-3))(q-2(m-1)). \qedend
		\end{align}
	\end{proof}
	By \cref{lem4.22}, 
	\begin{align}
		f_2(q) = \#X_q^\ell = \begin{cases*}
			(q-2)(q-6) \cdots (q-2(2\ell-3))(q-2(\ell-1)) 	& \textup{if $\gcd\{4.q\} = 2$;}\\
			q(q-4) \cdots (q-4(\ell-1))						& \textup{if $\gcd\{4,q\} = 4$.}
		\end{cases*}
	\end{align}
	Therefore, by the coset method (\cref{cosetmethod}), we have
	\begin{align}
		\chi_{2\PHI(B_\ell),\,Z}^\quasi(q) &= \frac{f_1(q) + f_2(q)}{2}\\
		&= \begin{cases*}
			(q-2)(q-6) \cdots (q-2(2\ell-3))(q-(3\ell-2))	& \textup{if $\gcd\{4,q\} = 2$;}\\
			(q-4)(q-8) \cdots (q-4(\ell-1))(q-\ell)			& \textup{if $\gcd\{4,q\} = 4$.}
		\end{cases*}
	\end{align}
	Now, the proof of \cref{prop2.9} is complete.
\end{proof}

\subsubsection{type $BC_\ell$}\ 

\noindent
Let $\{e_1,\ldots,e_\ell\}$ be a standard basis for $E = \mathbb{R}^\ell$.
Then
\begin{align}
	\PHI = \bigset{\pm(e_i-e_j),\ \pm(e_i+e_j),\ \pm e_k,\ \pm 2e_k}{i,j,k \in \{1,\ldots,\ell\},\ i < j}
\end{align}
is a root system of type $BC_\ell$.
The roots $\DELTA = \{e_1-e_2, \, \ldots, \, e_{\ell-1}-e_\ell, \, e_\ell\}$ is a basis of $\PHI$.
Then
\begin{align}
	\DELTA^\vee = \bigset{e_1 + \cdots + e_i}{i \in \{1,\ldots,\ell\}}.
\end{align}
Hence the coweight lattice $Z$ is isomorphic to the lattice $L$ generated by $\{e_1,\ldots,e_\ell\}$.

\begin{proposition}[type $BC_\ell$]\label{typeBCroot}
	Let $\PHI$ be a root system of type $BC_\ell$.
	Then
	\begin{align}
		\chi_{\PHI}^\quasi(q) = \begin{cases*}
			(q-1)(q-3) \cdots (q-(2\ell-1)) 		& \textup{if $q \not\in 2\mathbb{Z}$;}\\
			(q-2)(q-4) \cdots (q-2(\ell-1))(q-2\ell) & \textup{if $q \in 2\mathbb{Z}$}.
		\end{cases*}
	\end{align}
\end{proposition}
\begin{proof}
	It follows from \cref{typeCroot} since $\chi_{\PHI}^\quasi = \chi_{\PHI(C_\ell),\, L}^\quasi$.	
\end{proof}

\subsubsection{type $D_\ell$}\ 

\noindent
Let $\{e_1,\ldots,e_\ell\}$ be a standard basis for $E = \mathbb{R}^\ell$.
Then
\begin{align}
	\PHI = \bigset{\pm(e_i-e_j),\ \pm(e_i+e_j)}{i,j \in \{1,\ldots,\ell\},\ i < j}
\end{align}
is a root system of type $D_\ell$.
The roots $\DELTA = \{e_1-e_2, \, \ldots, \, e_{\ell-1}-e_\ell, \, e_{\ell-1} + e_\ell\}$ is a basis of $\PHI$.
Then
\begin{align}
	\DELTA^\vee = \bigset{e_1 + \cdots + e_i}{i \in \{1,\ldots,\ell-2\}} \sqcup \left\{ \frac{e_1 + \cdots + e_{\ell-1} - e_\ell}{2},\ \frac{e_1 + \cdots + e_\ell}{2} \right\}.
\end{align}

\begin{proposition}[type $D_\ell$, {\cite[Theorem 4.8]{KamiyaTakemuraTeraoRoot}}]\label{typeDroot}
	Let $\PHI$ be a root system of type $D_\ell$.
	Then
	\begin{align}
		\chi_\PHI^\quasi(q) = \begin{cases*}
			(q-1)(q-3) \cdots (q-(2\ell-3))(q-(\ell-1))  		& \textup{if $q \not\in 2\mathbb{Z}$;}\\
			(q-2)(q-4) \cdots (q-(2\ell-4))(q^2 - 2(\ell-1)q + \frac{\ell(\ell-1)}{2}) & \textup{if $q \in 2\mathbb{Z}$}.
		\end{cases*}
	\end{align}
	Let $L$ be a lattice generated by $\{e_1,\ldots,e_\ell\}$.
	Then
	\begin{align}
		\chi_{\PHI,L}^\quasi(q) = \begin{cases*}
			(q-1)(q-3) \cdots (q-(2\ell-3))(q-(\ell-1)) 		& \textup{if $q \not\in 2\mathbb{Z}$;}\\
			(q-2)(q-4) \cdots (q-(2\ell-4))(q^2 - 2(\ell-1)q + \ell(\ell-1)) & \textup{if $q \in 2\mathbb{Z}$}.
		\end{cases*}
	\end{align}
\end{proposition}
\begin{proof}
	The function $\chi_{\PHI,L}^\quasi$ is exactly equal to $\chi_{T(D_\ell)}$ in \cite[Proposition 4.3, Theorem 4.5]{KamiyaTakemuraTeraoRoot}
\end{proof}

\subsubsection{type $E_6$}\ 

\noindent
Let $\{e_1,\ldots,e_8\}$ be a standard basis for $\mathbb{R}^8$, and define
\begin{align}
	E \ceq \bigset{(x_1,\ldots,x_8) \in \mathbb{R}^8}{x_6  = x_7 = -x_8}.
\end{align}
Then
\begin{align}
	\PHI = \bigset{\pm(e_i-e_j),\ \pm(e_i+e_j)}{i,j \in \{1,\ldots,5\},\ i > j} \sqcup \Bigset{\pm\frac{\nu_1e_1 + \cdots + \nu_5e_5 - e_6 -e_7 + e_8}{2}}{
		\begin{lgathered}
			\nu_1,\ldots,\nu_5 \in \{-1,1\},\\ 
			\text{$\nu_1 \cdots \nu_5 = 1$}
		\end{lgathered}
	}
\end{align}
is a root system of type $E_6$.
The roots
\begin{align}
	\DELTA = \left\{\frac{e_1 - e_2 - e_3 - e_4 - e_5 - e_6 - e_7 + e_8}{2}, \ e_1 + e_2,\ e_2 - e_1,\ e_3 - e_2,\ e_4 - e_3,\ e_5 - e_4 \right\}	
\end{align}
is a basis of $\PHI$.

\begin{proposition}[type $E_6$, {\cite[\S6.1]{KamiyaTakemuraTeraoRoot}}]\label{typeE6root}
	Let $\PHI$ be a root system of type $E_6$.
	Then
	\begin{align}
		\chi_\PHI^\quasi(q) = \begin{cases*}
				(q-1)(q-4)(q-5)(q-7)(q-8)(q-11)				& \textup{if $\gcd\{6,q\} = 1$;}\\
				(q-2)(q-4)(q-8)(q-10)(q^2-12q+26)			& \textup{if $\gcd\{6,q\} = 2$;}\\
				(q-3)(q-9)(q^4-24q^3+195q^2-612q+480)		& \textup{if $\gcd\{6,q\} = 3$;}\\
				(q-6)^2(q^4-24q^3+186q^2-504q+480)			& \textup{if $\gcd\{6,q\} = 6$.}
		\end{cases*}
	\end{align}
\end{proposition}

\subsubsection{type $E_7$}\ 

\noindent
Let $\{e_1,\ldots,e_8\}$ be a standard basis for $\mathbb{R}^8$, and define
\begin{align}
	E \ceq \bigset{(x_1,\ldots,x_8) \in \mathbb{R}^8}{x_7 = -x_8}.
\end{align}
Then
\begin{align}
	\PHI &= \bigset{\pm(e_i-e_j),\ \pm(e_i+e_j)}{i,j \in \{1,\ldots,6\},\ i > j} \sqcup \bigl\{ \pm(e_8 - e_7) \bigr\} \\
	&\quad \qquad \qquad \sqcup \Bigset{\pm\frac{\nu_1e_1 + \cdots + \nu_6e_6 - e_7 + e_8}{2}}{
		\begin{lgathered}
			\nu_1,\ldots,\nu_6 \in \{-1,1\},\\ 
			\text{$\nu_1 \cdots \nu_6 = -1$}
		\end{lgathered}
	}
\end{align}
is a root system of type $E_7$.
The roots
\begin{align}
	\DELTA = \left\{\frac{e_1 - e_2 - e_3 - e_4 - e_5 - e_6 - e_7 + e_8}{2}, \ e_1 + e_2,\ e_2 - e_1,\ e_3 - e_2,\ e_4 - e_3,\ e_5 - e_4,\ e_6 - e_5 \right\}	
\end{align}
is a basis of $\PHI$.

\begin{proposition}[type $E_7$, {\cite[\S6.2]{KamiyaTakemuraTeraoRoot}}]\label{typeE7root}
	Let $\PHI$ be a root system of type $E_7$.
	Then
	\begin{align}
		\chi_\PHI^\quasi(q) = \begin{cases*}
			(q-1)(q-5)(q-7)(q-9)(q-11)(q-13)(q-17)						& \textup{if $\gcd\{12,q\} = 1$;}\\
			(q-2)(q-10)(q-13)(q-14)(q^3-24q^2+155q-342)					& \textup{if $\gcd\{12,q\} = 2$;}\\
			(q-3)(q-9)(q-15)(q^4-36q^3+438q^2-2052q+2289)				& \textup{if $\gcd\{12,q\} = 3$;}\\
			(q-4)(q-5)(q-8)(q-16)(q^3-30q^2+263q-504)					& \textup{if $\gcd\{12,q\} = 4$;}\\
			(q-6)(q^6-57q^5+1275q^4-14085q^3+79374q^2-213228q+234360)	& \textup{if $\gcd\{12,q\} = 6$;}\\
			(q-12)(q^6-51q^5+1005q^4-9675q^3+47784q^2-116064q+120960)	& \textup{if $\gcd\{12,q\} = 12$.}
		\end{cases*}
	\end{align}
\end{proposition}

We omit the result for type $E_8$ since it is not necessary in this paper.

\subsubsection{type $F_4$}\ 

\noindent
Let $\{e_1,e_2,e_3,e_4\}$ be a standard basis for $E = \mathbb{R}^4$.
Then
\begin{align}
	\PHI &= \bigset{\pm(e_i-e_j),\ \pm(e_i+e_j),\ \pm e_k}{i,j,k \in \{1,2,3,4\},\ i < j} \\
	&\qquad \qquad \sqcup \Bigset{\frac{\nu_1e_1 + \nu_2e_2 + \nu_3e_3 + \nu_4e_4}{2}}{
			\nu_1,\nu_2,\nu_3,\nu_4 \in \{-1,1\}
	}
\end{align}
is a root system of type $F_4$.
The roots
\begin{align}
	\DELTA = \left\{
		e_2 - e_3,\ e_3 - e_4,\ e_4,\ \frac{e_1 - e_2 - e_3 - e_4}{2}
	\right\}	
\end{align}
is a basis of $\PHI$.

\begin{proposition}[type $F_4$, {\cite[\S5.2]{KamiyaTakemuraTeraoRoot}}]\label{typeF4root}
	Let $\PHI$ be a root system of type $F_4$.
	Then
	\begin{align}
		\chi_{\PHI}^\quasi(q) = \begin{cases*}
			(q-1)(q-5)(q-7)(q-11)		& \textup{if $\gcd\{12,q\} = 1$;}\\
			(q-2)(q-10)(q^2-12q+44)		& \textup{if $\gcd\{12,q\} = 2$;}\\
			(q-3)(q-9)(q^2-12q+19)		& \textup{if $\gcd\{12,q\} = 3$;}\\
			(q-4)^2(q-8)^2				& \textup{if $\gcd\{12,q\} = 4$;}\\
			(q-6)^2(q^2-12q+28)			& \textup{if $\gcd\{12,q\} = 6$;}\\
			q^4-24q^3+208q^2-768q+1152	& \textup{if $\gcd\{12,q\} = 12$.}
		\end{cases*}
	\end{align}
	Let $\PHI^\vee$ denote the dual root system of $\PHI$.
	Then
	\begin{align}
		\chi_{\PHI^\vee,\, Z}^\quasi(q) = \begin{cases*}
			(q-1)(q-5)(q-7)(q-11)		& \textup{if $\gcd\{12,q\} = 1$;}\\
			(q-2)(q-10)^2(q-14)			& \textup{if $\gcd\{12,q\} = 2$;}\\
			(q-3)(q-9)(q^2-12q+19)		& \textup{if $\gcd\{12,q\} = 3$;}\\
			(q-4)(q-8)^2(q-16)			& \textup{if $\gcd\{12,q\} = 4$;}\\
			(q-6)(q^3-30q^2+268q-552)	& \textup{if $\gcd\{12,q\} = 6$;}\\
			(q-12)(q^3-24q^2+160q-384)	& \textup{if $\gcd\{12,q\} = 12$.}\\
		\end{cases*}
		\label{F4vee}
	\end{align}
\end{proposition}
\begin{proof}
	The cosfficient matrix $S'$ of $\PHI^\vee$ with respect to $Z$ is as follows:
	\begin{align}
		S' = \left(\begin{array}{cccccccccccccccccccccccc}
			1 & 0 & 0 & 0 & 1 & 0 & 0 & 2 & 0 & 0 & 2 & 1 & 0 & 2 & 1 & 0 & 1 & 2 & 1 & 2 & 2 & 1 & 1 & 2 \\
			0 & 1 & 0 & 0 & 1 & 2 & 0 & 2 & 2 & 1 & 2 & 1 & 2 & 2 & 2 & 1 & 1 & 4 & 2 & 4 & 4 & 2 & 3 & 3 \\
			0 & 0 & 2 & 0 & 0 & 2 & 2 & 2 & 2 & 2 & 2 & 2 & 4 & 4 & 2 & 2 & 2 & 4 & 2 & 6 & 6 & 4 & 4 & 4 \\
			0 & 0 & 0 & 2 & 0 & 0 & 2 & 0 & 2 & 0 & 2 & 0 & 2 & 2 & 0 & 2 & 2 & 2 & 2 & 2 & 4 & 2 & 2 & 2
		\end{array}\right).
	\end{align}
	This matrix is obtained by multiplying by $2$ the columns corresponding to short roots of the coefficient matrix of $\PHI$.
	Therefore it is easy to see that $\chi_{\PHI^\vee,\, Z}^\quasi(q) = \chi_\PHI^\quasi(q)$ if $q$ is odd.
	Moreover, $\chi_\PHI^\quasi(q) = 0$ implies $\chi_{\PHI^\vee,\, Z}^\quasi(q) = 0$, that is, $\chi_{\PHI^\vee,\, Z}^\quasi(q) = 0$ if $q < 12$.
	
	Use \cref{consticoef} for $\chi_{\PHI^\vee,\, Z}^\quasi$.
	We can see that 
	\begin{align}
		\mathcal{E}_1 = \mathcal{E}_2 = \{1,2\},\qquad \mathcal{E}_3 = \{1,2,4\},\qquad \mathcal{E}_4 = \{1,2,4,6\}.
	\end{align}
	Hence the minimum period of $\chi_{\PHI^\vee,\, Z}^\quasi$ is $\lcm\mathcal{E}_4 = 12$.
	For each divisor $r$ of $12$, let $f^{(r)}$ be the $r$-th constituent of $\chi_{\PHI^\vee,\, Z}^\quasi$.
	Then we have
	\begin{align}
		\deg(f^{(r_1)} - f^{(r_2)}) < 2
	\end{align}
	for all $r_1,r_2 \in \{2,4,6,12\}$, and 
	\begin{align}
		\deg(f^{(2)} - f^{(6)}) < 1,\qquad \deg(f^{(4)} - f^{(12)}) < 1.
	\end{align}
	Therefore it is sufficient to compute the special value
	\begin{align}
		\chi_{\PHI^\vee,\,Z}^\quasi(12) = \chi_{\PHI'^\vee,\,Z}^\quasi(14) = \chi_{\PHI^\vee,\,Z}^\quasi(16) = 0
	\end{align}
	to obtain \cref{F4vee}.
\end{proof}

\subsubsection{type $G_2$}\ 

\noindent 
Let $\{e_1,e_2,e_3\}$ be a standard basis for $\mathbb{R}^3$, and define
\begin{align}
	E \ceq \bigset{(x_1,x_2,x_3) \in \mathbb{R}^3}{x_1 + x_2 + x_3 = 0}.
\end{align}
Then
\begin{align}
	\PHI &= \Bigl\{
		\pm(e_1 - e_2),\ 
		\pm(-2e_1 + e_2 + e_3),\ 
		\pm(e_3 - e_1),\ 
		\pm(e_3 - e_2).\ 
		\pm(e_1 - 2e_2 + e_3),\ 
		\pm(-e_1 - e_2 + 2e_3)
	\Bigr\}
\end{align}
is a root system of type $G_2$.
The roots
\begin{align}
	\DELTA = \left\{
		e_1 - e_2,\ -2e_1 + e_2 + e_3
	\right\}	
\end{align}
is a basis of $\PHI$.

\begin{proposition}[type $G_2$, {\cite[\S5.1]{KamiyaTakemuraTeraoRoot}}]\label{typeG2root}
	Let $\PHI$ be a root system of type $G_2$.
	Then
	\begin{align}
		\chi_{\PHI}^\quasi(q) = \begin{cases*}
			(q-1)(q-5)		& \textup{if $\gcd\{6,q\} = 1$;}\\
			(q-2)(q-4)		& \textup{if $\gcd\{6,q\} = 2$;}\\
			(q-3)^2			& \textup{if $\gcd\{6,q\} = 3$;}\\
			q^2-6q+12			& \textup{if $\gcd\{6,q\} = 6$.}
		\end{cases*}
	\end{align}
	Let $\PHI^\vee$ denote the dual root system of $\PHI$.
	Then
	\begin{align}
		\chi_{\PHI^\vee,\,Z}^\quasi(q) = \begin{cases*}
			(q-1)(q-5)		& \textup{if $\gcd\{6,q\} = 1$;}\\
			(q-2)(q-4)		& \textup{if $\gcd\{6,q\} = 2$;}\\
			(q-3)(q-9)		& \textup{if $\gcd\{6,q\} = 3$;}\\
			(q-6)^2			& \textup{if $\gcd\{6,q\} = 6$.}
		\end{cases*}
		\label{G2vee}
	\end{align}
\end{proposition}
\begin{proof}
	The cosfficient matrix $S'$ of $\PHI^\vee$ with respect to $Z$ is as follows:
	\begin{align}
		S' = \left(\begin{array}{cccccc}
			3 & 0 & 3 & 6 & 3 & 3\\
			0 & 1 & 3 & 3 & 1 & 2
		\end{array}\right).
	\end{align}
	This matrix is obtained by multiplying by $3$ the columns corresponding to short roots of the coefficient matrix of $\PHI$.
	Therefore it is easy to see that $\chi_{\PHI^\vee,\, Z}^\quasi(q) = \chi_\PHI^\quasi(q)$ if $q \not\in 3\mathbb{Z}$.
	Moreover, $\chi_\PHI^\quasi(q) = 0$ implies $\chi_{\PHI^\vee,\, Z}^\quasi(q) = 0$, that is, $\chi_{\PHI^\vee,\, Z}^\quasi(q) = 0$ if $q < 6$.
	
	Use \cref{consticoef} for $\chi_{\PHI^\vee,\, Z}^\quasi$.
	We can see that 
	\begin{align}
		\mathcal{E}_1 = \{1,3\},\qquad \mathcal{E}_2 = \{1,3,6\}.
	\end{align}
	Hence the minimum period of $\chi_{\PHI^\vee,\, Z}^\quasi$ is $\lcm\mathcal{E}_2 = 6$.
	For each divisor $r$ of $6$, let $f^{(r)}$ be the $r$-th constituent of $\chi_{\PHI^\vee,\, Z}^\quasi$.
	Then we have
	\begin{align}
		\deg(f^{(3)} - f^{(6)}) < 1.
	\end{align}
	Therefore it is sufficient to compute the special value
	\begin{align}
		\chi_{\PHI^\vee,\, Z}^\quasi(6) = \chi_{\PHI^\vee,\, Z}^\quasi(9) = 0
	\end{align}
	to obtain \cref{G2vee}.
\end{proof}

\section{Equivariant version of the characteristic quasi-polynomials}\label{sec3}

We discuss the equivariant version of the characteristic quasi-polynomials of root systems.
For details, including the general case, see \cite{Uchiumi}.

\subsection{Notations}\ \label{sec3.1}

\noindent
Let $\PHI$ be an irreducible reduced root system.
Recall that 
\begin{align}
	M(\PHI;q) \ceq M(\PHI,Z;q) = \bigset{\pi_q(x) \in Z/qZ}{(\beta,x) \not\equiv 0 \pmod{q} \tforall \beta \in \PHI^+}
\end{align}
for the coweight lattice $Z$ of $\PHI$.
Consider the action of $W$ on $M(\PHI; q)$.
Since $Z$ is invariant under the action of $W$ on $E$, the action $\rho_q : W \lra \GL(Z/qZ)$ is induced as follows:
\begin{align}
	\rho_q(w) : \pi_q(x) \lmapsto \pi_q(w(x))
\end{align}
for all $w \in W$ and $x \in Z$.
We will omit the notation $\rho_q$ below for simplicity.
Since $\PHI$ is invariant under the action of $W$, the set $M(\PHI; q)$ is also invariant under $W$ \cite[Lemma 2.3]{Uchiumi}.
Hence we can consider the permutation character $\chi_{\PHI,q}$ on $M(\PHI; q)$ for each $q \in \mathbb{Z}_{>0}$.
As well known, for any $w \in W$, we have
\begin{align}
	\chi_{\PHI,q}(w) = \#\bigset{\pi_q(x) \in Z/qZ}{w(\pi_q(x)) = \pi_q(x)}.
\end{align}
By substituting the identity element $1 \in W$ into $\chi_{\PHI,q}$, we obtain the characteristic quasi-polynomial:
\begin{align}
	\chi_{\PHI,q}(1) = \chi_\PHI^\quasi(q).
\end{align}

In \cite{Uchiumi}, the following is obtained.
\begin{theorem}[{\cite[Theorem 2.6]{Uchiumi}}]
	Let $w \in W$.
	Then $\chi_{\PHI,q}(w)$ is a quasi-polynomial in $q$, and has the gcd-property.
\end{theorem}
One of the purposes of this paper is to compute all the constituents of $\chi_{\PHI,q}(w)$.

For later, reinterpret the characteristic quasi-polynomial as a counting of $q$-torsion points on the torus.
Let $T \ceq E/Z$ be an $\ell$-torus, and $\pi_T : E \lra T$ be the natural projection.
We consider the following set as the complement of $\A_\PHI$ in $T$:
\begin{align}
	T(\A_\PHI) \ceq \pi_T(M^\aff) = \bigset{\pi_T(x) \in T}{(\beta,x) \not\in \mathbb{Z} \tforall \beta \in \PHI^+}.
\end{align}

For $q \in \mathbb{Z}_{>0}$, let $T[q]$ be a set of $q$-torsion points of $T$:
\begin{align}
	T[q] \ceq \bigset{t \in T}{qt = 0} = \bigset{\pi_T(x) \in T}{qx \in L}.
\end{align}
There exists a bijection $f : T[q] \lra Z/qZ$ defined by
\begin{align}
	f : \pi_T(x) \lmapsto \pi_q(qx)
\end{align}
for all $x \in \pi_T^{-1}(T[q])$.
Let $\GL(T)$ be a group of linear transformations of $T$ defined by
\begin{align}
	\GL(T) = \bigset{g :T \lra T}{\text{there exists $g' \in \GL(E)$ such that $g \circ \pi_T = \pi_T \circ g'$}}.
\end{align}
Then the action $\rho_T : W \lra \GL(T)$ is induced as follows:
\begin{align}
	\rho_T(w) : \pi_T(x) \lmapsto \pi_T(w(x)).
\end{align}
for all $w \in W$ and $x \in \pi_T^{-1}(T[q])$.
We will omit the notation $\rho_T$ below for simplicity.
It follows from \cite[Proposition 3.1]{Uchiumi} that the bijection $f$ is $W$-equivalent, that is, 
\begin{align}
	f \circ \rho_T(w) = \rho_q(w) \circ f
\end{align}
for all $w \in W$.
Furthermore, we have $f(T[q] \cap T(\A_\PHI)) = M(\PHI;\, q)$ \cite[Lemma 3.6]{Uchiumi}.
In other words, $T[q] \cap T(\A_\PHI)$ is isomorphic to $M(\PHI;\, q)$ as a $W$-set.
Therefore $\chi_{\PHI,q}$ is equal to the permutation character on $T[q] \cap T(\A)$, that is, 
\begin{align}
	\chi_{\PHI,q}(w) = \#\bigset{t \in T[q] \cap T(\A)}{w(t) = t}
\end{align}
for all $w \in W$.

\subsection{Representation using Ehrhart quasi-polynomials}\ \label{sec3.2}

\noindent
An approach to compute $\chi_{\PHI,q}(w)$ is to determine the equivariant Ehrhart quasi-polynomial of $A_\circ$.
Let $P^\diamond$ denote a fundamental parallelepiped of $Z$ defined by
\begin{align}
	P^\diamond = \sum_{i=1}^\ell (0,1] \varpivee_i = \bigset{x \in E}{0 < (\alpha_i,x) \leq 1 \tforall i \in \{1,\ldots,\ell\}}.
\end{align}
Then the restriction map $\pi_T^\diamond : P^\diamond  \lra T$ of $\pi_T$ is a bijection, and the action $\rho^\diamond$ of $W$ on $P^\diamond$ is induced by
\begin{align}
	\rho^\diamond(w) &= (\pi_T^\diamond)^{-1} \circ \rho_T(w) \circ \pi_T^\diamond\\
	 &= (\pi_T^\diamond)^{-1} \circ \pi_T \circ w
\end{align}
for all $w \in W$.
In other words, for any $x \in P^\diamond$, we have
\begin{align}
	\rho^\diamond(w)(x) = w(x) + z_{w,x}, \label{zwx}
\end{align}
where $z_{w,x} \in Z$ is the unique element such that $w(x) + z_{w,x} \in P^\diamond$.
Let $\mathcal{C}^\diamond$ denote the set of alcoves contained in $P^\diamond$.
In particular, the fundamental alcove $A_\circ$ belongs to $\mathcal{C}^\diamond$.
Then 
\begin{align}
	P^\diamond \setminus \bigcup_{H \in \A^\aff_\PHI}H = \bigsqcup_{C \in \mathcal{C}^\diamond}C.
\end{align}
Hence we have
\begin{align}
	\chi_\PHI^\quasi(q)	&= \sum_{C \in \mathcal{C}^\diamond}\Ell_C(q). \label{sumofEpoly}
\end{align}

\begin{remark}
	It is known in \cite[\S4.9]{Humphreys} that 
	\begin{align}
		\frac{\vol{P^\diamond}}{\vol{A_\circ}} = \frac{\#W}{f} = \ell! \cdot n_1 \cdots n_\ell.
	\end{align}
	Since all alcoves have the same Ehrhart quasi-polynomials, \cref{yoshinagathm} follows from \cref{sumofEpoly}.
	In particular, in the case where $q$ is coprime to $n_i$ for all $i \in \{1,\ldots,n\}$ (i.e., $q$ is coprime to $\tilde{n}_\PHI$), Athanasiadis proved the following \cite[Theorem 3.1]{Athanasiadis2004}:
	\begin{align}
		\chi_{\A_\PHI^m}(q) = \frac{\#W}{f} \cdot \Ell_{A_\circ}(q-mh),
	\end{align}
	where $\A_\PHI^m$ is a deformation of the Coxeter arrangement defined by
	\begin{align}
		\A_\PHI^m \ceq \bigset{H_\beta^k = \Bigset{x \in E}{(\beta,x) = k}}{\beta \in \PHI,\ k \in \{0,\ldots,m\}}
	\end{align}
	for $m \in \mathbb{Z}_{\geq 0}$.
	
\end{remark}

For each alcove $C \in \mathcal{C}^\diamond$, let $W_{C}$ and $W(C)$ denote the stabilizer subgroup and the orbit of $C$:
\begin{align}
	W_C = \bigset{w \in W}{\rho^\diamond(w)(C) = C},\qquad
	W(C) = \bigset{\rho^\diamond(w)(C) \in \mathcal{C}^\diamond}{w \in W}.
\end{align}
Let $C^w$ denote the set of points in $C$ fixed by $\rho^\diamond(w)$:
\begin{align}
	C^w = \bigset{x \in C}{\rho^\diamond(w)(x) = x}.
\end{align}
The \textbf{equivariant Ehrhart quasi-polynomial} $\chi_{C,q}$ of $C$ is the permutation character of $W_C$ on $qC \cap Z$, that is, we have
\begin{align}
	\chi_{C,q}(w) &= \#\bigset{x \in qC \cap Z}{\rho^\diamond(w)(x) = x}\\
	&= \#(qC^w \cap Z) = \Ell_{C^w}(q)
\end{align}
for all $w \in W_C$.
In \cite{Uchiumi}, we obtained the equivalent version of \cref{yoshinagathm}.
Let $K$ be a subgroup of $W$.
For any character $\chi : K \lra \mathbb{C}$ of $H$, the character $\Ind^W_K\chi$ of $W$ is induced such that
\begin{align}
	\left(\Ind^W_K\chi\right)(w) = \frac{1}{\#K}\sum_{\substack{u \in W\\ u^{-1}wu \in K}}\chi(u^{-1}wu)
\end{align}
for all $w \in W$.

\begin{theorem}[{\cite[Theorem 4.3, Theorem 5.2]{Uchiumi}}]\label{thm4.2}
	For $q \in \mathbb{Z}_{>0}$, we have
	\begin{align}
		\chi_{\PHI,q} = \sum_{C \in \mathcal{C}^\diamond}\frac{1}{\#W(C)}\Ind^W_{W_C}\chi_{C,q} = \Ind^W_{W_{A_\circ}}\chi_{A_\circ,q}. \label{equivyosthm}
	\end{align}
\end{theorem}

By substituting the identity element $1 \in W$ into \cref{equivyosthm}, it recovers \cref{yoshinagathm}.
In \cite{Uchiumi}, the above formula is actually used to compute $\chi_{\PHI,q}$ for type $A_\ell$.

\begin{corollary}
	If $w \in W$ is not conjugate to any element of $W_{A_\circ}$, then $\chi_{\PHI,q}(w) = 0$ for all $q \in \mathbb{Z}_{>0}$.
\end{corollary}
It is sufficient to compute $\chi_{\PHI,q}(w)$ only when $w \in W_{A_\circ}$.

\subsection{Stabilizer subgroups and Extended affine Weyl group}\ \label{sec3.3}

\noindent
In this subsection, we will describe the stabilizer subgroup with respect to the fundamental alcove.
For detail, see \cite{Garnier, KomrakovPremet}.

The semidirect group $\widehat{W_\aff} \ceq Z \rtimes W$ is called the \textbf{extended affine Weyl group} of $\PHI$.
It is clear that $W_\aff$ is a subgroup of $\widehat{W_\aff}$, but $\widehat{W_\aff}$ is not a reflection group.
Moreover, $\widehat{W_\aff}$ acts transively on $\mathcal{C}(\A_\PHI^\aff)$ but not simply transively, that is, $\overline{A_\circ}$ is not a fundamental domain for the action of $\widehat{W_\aff}$.
For any $z \in Z$, let $t_{z}$ denote a translation map on $E$ defined by $t_{z} : x \lmapsto x + z$.

Let $\widehat{\OMEGA}$ denote the stabilizer subgroup of $\widehat{W_\aff}$ with respect to the fundamental alcove:
\begin{align}
	\widehat{\OMEGA} \ceq \bigset{\widehat{\omega} \in \widehat{W_\aff}}{\widehat{\omega}(A_\circ) = A_\circ}.
\end{align}
It is well known that $\widehat{W_\aff} \cong W_\aff \rtimes \widehat{\OMEGA}$, and 
\begin{align}
	\widehat{\OMEGA} \cong \widehat{W_\aff}/W_\aff \cong Z/\veeQ.
\end{align}

Let $I \ceq \{0,\ldots,\ell\}$, and set
\begin{align}
	J \ceq \bigset{i \in I}{n_i = 1}.
\end{align}
Let $w_0$ denote the longest element of $W$.
For $j \in J \setminus \{0\}$, let $w_j$ be the longest element of the Weyl group corresponding to $\DELTA \setminus \{\alpha_j\}$.
We define $\omega_j \in W$ and $\widehat{\omega_j} \in \widehat{W_\aff}$ by
\begin{align}
	\omega_j \ceq w_jw_0,\qquad
	\widehat{\omega_j} \ceq t_{\varpivee_j} \circ \omega_j.
\end{align}
Then we have the following {\cite[VI, \S2, Proposition 6]{Bourbaki}}, {\cite[Proposition-Definition 2.2.1]{Garnier}}:
\begin{align}
	\widehat{\OMEGA} = \bigset{\widehat{\omega_j} \in \widehat{W_\aff}}{j \in J}.
\end{align}
The group $\widehat{\OMEGA}$ acts on the vertex set $\tbigset{\frac{\varpivee_i}{n_i}}{i \in I}$ of $A_\circ$ ({\cite[Lemma 2.2.4]{Garnier}}).
Define a permutation $\sigma_j$ on $I$ such that 
\begin{align}
	\widehat{\omega_j}\left(\dfrac{\varpivee_i}{n_i}\right) = \frac{\varpivee_{\sigma_j(i)}}{n_{\sigma_j(i)}}
\end{align}
for all $i \in I$.
Then it follows from {\cite[Lemma 2.2.5]{Garnier} and \cite[Lemma 1]{KomrakovPremet}} that 
\begin{iitemize}
	\item $\omega_j(\alpha_i) = \alpha_{\sigma_j(i)}$ for all $i \in I$;
	\item $n_i = n_{\sigma_j(i)}$ for all $i \in I$;
	\item $\sigma_j(0) = j$.
\end{iitemize}
Furthermore, by definition, we have
\begin{align}
	\omega_j^t(\varpivee_i) = \varpivee_{\sigma_j^t(i)} - n_i\varpivee_{\sigma_j^t(0)} \label{prop3.4}
\end{align}
for all $t \in \mathbb{Z}$.

Define a set
\begin{align}
	\OMEGA \ceq \bigset{\omega_j \in W}{j \in J},
\end{align}
and let $\pi : \widehat{W_\aff} \lra W$ be the projection.
Then $\pi(\widehat{\OMEGA}) = \OMEGA$.
Hence we can see that $\OMEGA$ is a subgroup of $W$ isomorphic to $\widehat{\OMEGA}$.

\begin{proposition}[{\cite[Lemma 2.2.5, Corollary 2.2.7]{Garnier}}]
	The group $\OMEGA$ acts on $\DELTA \cup \{\alpha_0\}$ and can be regarded as a subgroup of automorphisms of the extended Dynkin diagram.
\end{proposition}

We can see that the stabilizer subgroup $W_{A_\circ}$ under the action $\rho^\diamond$ is exactly equal to $\OMEGA$:
\begin{proposition}\label{prop4.6}
	$W_{A_\circ} = \OMEGA = \tbigset{\omega_j \in W}{j \in J}$.
\end{proposition}
\begin{proof}
	Let $\omega_j \in \OMEGA$.
	Since $\omega_j(A_\circ) = A_\circ - \varpivee_j$, we have 
	\begin{align}
		\rho^\diamond(\omega_j)(A_\circ) = ((\pi_T^\diamond)^{-1} \circ \pi_T \circ \omega_j)(A_\circ) = (\pi_T^\diamond)^{-1}(\pi_T(A_\circ)) = A_\circ.
	\end{align}
	Hence $\omega_j \in W_{A_\circ}$.
	
	Suppose that $w \in W_{A_\circ}$.
	Then
	\begin{align}
		((\pi_T^\diamond)^{-1} \circ \pi_T \circ w)(A_\circ) = \rho^\diamond(w)(A_\circ) = A_\circ.
	\end{align}
	It follows from the above that there exists $z \in Z$ such that $w(A_\circ) = A_\circ - z$.
	Therefore $t_z \circ w$ belongs to $\widehat{W_\aff}$.
	Hence there exists $j \in J$ such that $w = \omega_j$ and $z = \varpivee_j$, that is, $w \in \OMEGA$.
\end{proof}

Let $x = \sum_{i=0}^\ell x_i\varpivee_i \in P^\diamond$, where we assume that $x_0 = 1$.
Let $\omega \ceq \omega_j \in \OMEGA = W_{A_\circ}$.
By \cref{prop3.4}, we have
\begin{align}
	\omega(x) = \sum_{i=0}^\ell x_i\omega(\varpivee_i) = \sum_{i=0}^\ell x_i\varpivee_{\sigma_j(i)} - \sum_{i=0}^\ell x_in_i\varpivee_j.
\end{align}
Since $\sum_{i=0}^\ell x_i\varpivee_{\sigma_j(i)} \in P^\diamond$, the element $z_{\omega,x}$ defined in \cref{zwx} is as follows:
\begin{align}
	z_{\omega,x} = \sum_{i=0}^\ell x_in_i\varpivee_j. \label{zwxx}
\end{align}

\subsection{Duality of the equivariant version of characteristic quasi-polynomials of root systems}\ \label{sec4.3}\label{sec3.4}

\noindent
We will show the equivariant version of \cref{thm2.2} \cref{thm2.2.3}.

For $w \in W$, let $r(w)$ be the dimension of the subspace of $E$ fixed by $w$.
Then $\ell - r(w)$ is equal to the minimum number of reflections required to express $w$ as a product of reflections.
Define a function $\delta : W \lra \mathbb{C}$ by 
\begin{align}
	\delta(w) \ceq (-1)^{\ell - r(w)}.
\end{align}
Then we can see that $\delta$ is a characer of $W$.

The following is  one of the main results of this paper.
\begin{theorem}\label{thm4.6}
	Let $\PHI$ be an irreducible reduced root system with Coxeter number $h$.
	Then
	\begin{align}
		\chi_{\PHI,q} = (-1)^\ell \cdot \delta \cdot \chi_{\PHI,h-q}.
	\end{align}
\end{theorem}
\begin{proof}
	For each $q \in \mathbb{Z}_{>0}$, let $F_0,F_1,\ldots,F_\ell$ be hyperplanes defined by
	\begin{align}
		F_0 \ceq H_{\tilde{\alpha}}^q,\quad 
		F_1 \ceq H_{\alpha_1}^0,\quad \ldots,\quad F_\ell \ceq H_{\alpha_\ell}^0.
	\end{align}
	Then each $F_i$ is a wall of $q\overline{A_\circ}$, that is, one of the facets of $q\overline{A_\circ}$ is included in $F_i$.
	
	To prove \cref{thm2.2} \cref{thm2.2.3}, for $q > h$, Yoshinaga gave the bijection $\varphi : (q\overline{A_\circ} \cap Z)\setminus \bigcup_{i=0}^\ell F_i \lra (q - h)\overline{A_\circ} \cap Z$ defined as follows \cite[Lemma 3.3]{Yoshinaga}:
	\begin{align}
		\varphi : x \lmapsto x - \sum_{i = 0}^\ell \varpivee_i.
	\end{align}
	For any $w \in W_{A_\circ}$, we consider the restriction $\varphi|_w$ of $\varphi$ to $(q\overline{A_\circ^w} \cap Z)\setminus \bigcup_{i=0}^\ell F_i$.
	
	\begin{lemma}\label{lem3.8}
		Let $w \in W_{A_\circ}$.
		Then 
		\begin{align}
			\varphi|_w\left( (q\overline{A_\circ^w} \cap Z)\setminus \bigcup_{i=0}^\ell F_i \right) = (q - h)\overline{A_\circ^w} \cap Z.
		\end{align}
		Hence $\varphi|_w : (q\overline{A_\circ^w} \cap Z)\setminus \bigcup_{i=0}^\ell F_i \lra (q - h)\overline{A_\circ^w} \cap Z$ is a bijection.
	\end{lemma}
	\begin{proof}[Proof of \cref{lem3.8}]
		Suppose that $w = \omega_j$.
		Let $x \in (q\overline{A_\circ^w} \cap Z)\setminus \bigcup_{i=0}^\ell F_i$, then we can set $x = qx'$ for $x' \in \overline{A_\circ^w}$.
		Since $w(A_\circ) + \varpivee_j = \widehat{\omega_j}(A_\circ) = A_\circ$, we can see that $z_{w,x'} = \varpivee_j$, that is, $x' = \rho^\diamond(w)(x') = w(x') + \varpivee_j$.
		Using \cref{zwxx}, we have
		\begin{align}
			w\left(\sum_{i=0}^\ell\varpivee_i\right) = \sum_{i=0}^\ell \varpivee_{\sigma_j(i)} - \sum_{i=0}^\ell n_i \varpivee_j = \sum_{i=0}^\ell \varpivee_i - h \varpivee_j.
		\end{align}
		Hence
		\begin{align}
			w\left( \frac{1}{q-h}\varphi|_w(x) \right) 
			&= \frac{1}{q-h}\left(q \cdot w(x') - w\left( \sum_{i=0}^\ell \varpivee_i \right)\right)\\
			&= \frac{1}{q-h}\left( x - q\varpivee_j - \sum_{i=0}^\ell \varpivee_j + h\varpivee_j \right)\\
			&= \frac{1}{q-h}\varphi|_w(x) - \varpivee_j.
		\end{align}
		The above implies that 
		\begin{align}
			\rho^\diamond(w)\left( \frac{1}{q-h}\varphi|_w(x) \right) = w\left( \frac{1}{q-h}\varphi|_w(x) \right) + \varpivee_j = \frac{1}{q-h}\varphi|_w(x).
		\end{align}
		Thus we have $\varphi|_w(x) \in (q-h)\overline{A_\circ^w} \cap Z$.
		
		Conversely, let $y \in (q-h)\overline{A_\circ^w} \cap Z$, and $y' \ceq \frac{y}{q-h}$.
		It is clear that $\varphi|_w(\varphi^{-1}(y)) = y$.
		Since $y' \in \overline{A_\circ^w}$, we have $y' = \rho^\diamond(w)(y') = w(y') + \varpivee_j$.
		Therefore
		\begin{align}
			w\left( \frac{1}{q}\varphi^{-1}(y) \right) 
			&= \frac{1}{q}\left( (q-h) \cdot w(y') + w\left( \sum_{i=0}^\ell \varpivee_i \right) \right)\\
			&= \frac{1}{q}\left( y - (q-h)\varpivee_j + \sum_{i=0}^\ell \varpivee_j - h\varpivee_j \right)\\
			&= \frac{1}{q}\varphi^{-1}(y) - \varpivee_j.
		\end{align}
		The above implies that 
		\begin{align}
			\rho^\diamond(w)\left( \frac{1}{q}\varphi^{-1}(y) \right) = w\left( \frac{1}{q}\varphi^{-1}(y) \right) + \varpivee_j = \frac{1}{q}\varphi^{-1}(y).
		\end{align}
		Thus we have $\varphi^{-1}(y) \in (q\overline{A_\circ^w} \cap Z) \setminus \bigcup_{i=0}^\ell F_i$.
	\end{proof}
	
	The bijection $\varphi|_w$ implies that 
	\begin{align}
		\#\left( (q\overline{A_\circ^w} \cap Z)\setminus \bigcup_{i=0}^\ell F_i \right)
		&= \#\Bigl( (q-h)\overline{A_\circ^w} \cap Z \Bigr).
	\end{align}
	Hence we have
	\begin{align}
		\chi_{A_\circ,q}(w) &= \Ell_{A^w_\circ}(q)\\
		&= \#\left( (q\overline{A_\circ^w} \cap L)\setminus \bigcup_{i=0}^\ell F_i \right)\\
		&= \#\Bigl( (q-h)\overline{A_\circ^w} \cap L \Bigr)\\
		&= \Ell_{\overline{A_\circ^w}}(q-h)\\
		&= (-1)^{ r(w)} \cdot \Ell_{A_\circ^w}(h-q)\\
		&= (-1)^{ r(w)} \cdot \chi_{A_\circ,h-q}(w).
	\end{align}
	for all $w \in W_{A_\circ}$.
	Therefore
	\begin{align}
		\chi_{\PHI,q}(w) &= \Ind^W_{W_{A_\circ}} \chi_{A_\circ,q}(w)\\
		&= \dfrac{1}{\#W_{A_\circ}} \sum_{\substack{u \in W\\ u^{-1}wu \in W_{A_\circ}}} \chi_{A_\circ,q}(u^{-1}wu)\\
		&= \dfrac{1}{\#W_{A_\circ}} \sum_{\substack{u \in W\\ u^{-1}wu \in W_{A_\circ}}} (-1)^{ r(w)} \cdot \chi_{A_\circ,h-q}(u^{-1}wu)\\
		&= (-1)^{ r(w)} \cdot \Ind^W_{W_{A_\circ}}\chi_{A_\circ,h-q}(w)\\
		&= (-1)^\ell \cdot \delta(w) \cdot \chi_{\PHI,h-q}(w)
	\end{align}
	for all $w \in W$.
	Now, the proof of \cref{thm4.6} is complete.
\end{proof}

\section{Relationship with root systems constructed by folding} \label{sec4}
\subsection{Decomposition using elementary divisors}\ \label{sec4.1}

\noindent
For $w \in W$, let $T^w$ and $T[q]^w$ denote the set of points of $T$ and $T[q]$ fixed by $\rho_T(w)$, respectively.
Then 
\begin{align}
	\pi_T^{-1}(T^w) = \bigset{x \in E}{(w - \id)(x) = w(x) - x \in Z}.
\end{align}
Let $R_w$ be the representation matrix of $w$, and $I$ denote the identity matrix.
In the elementary divisors theory, there exist $\ell \times \ell$ unimodular matrices $U,V$ and non-negative integers $d_1,\ldots,d_\ell \in \mathbb{Z}_{\geq 0}$ such that 
\begin{align}
	U(R_w-I)V^{-1} = \diag(d_1,\ldots,d_\ell).
\end{align}
Note that $r(w)$ defined in \cref{sec4.3} is equal to the number of $d_i$'s that are zero.
If $d_1,\ldots,d_\ell$ satisfy
\begin{align}
	d_{i_1},\ldots, d_{i_{\ell - r(w)}} \neq 0,\quad 
	d_{i_1} \mid d_{i_2} \mid \cdots \mid d_{i_{\ell - r(w)}},\quad
	d_{i_{\ell - r(w)+1}} = \cdots = d_{i_{\ell}} = 0,
\end{align}
then $\{d_{i_1},\ldots,d_{i_{\ell - r(w)}}\}$ is called the \textbf{elementary divisors} of the matrix $R_w - I$.
For each $i \in \{1,\ldots,\ell\}$, define $u_i,v_i \in Z$ as 
\begin{align}
	u_i \ceq \sum_{j=1}^\ell u_{ij}\varpivee_j,\qquad
	v_i \ceq \sum_{j=1}^\ell v_{ij}\varpivee_j,
\end{align}
where $U = (u_{ij})_{ij}$ and $V = (v_{ij})_{ij}$.
Then we can see that
\begin{align}
	w(u_i) = u_i + d_iv_i
\end{align}
for all $i \in \{1,\ldots,\ell\}$.

\begin{lemma}[{\cite[Lemma 3.2, Theorem 3.3]{Uchiumi}}]
	Using the notation above, we have
	\begin{align}
		\pi_T^{-1}(T^w) = \bigoplus_{\substack{1 \leq i \leq \ell\\d_i \neq 0}}d_i^{-1}\mathbb{Z}u_i \oplus \bigoplus_{\substack{1 \leq i \leq \ell\\ d_i = 0}}\mathbb{R}u_i.
	\end{align}
	Hence, for $q \in \mathbb{Z}_{>0}$,
	\begin{align}
		T^w = \Bigset{ \sum_{\substack{1 \leq i \leq \ell\\ d_i \neq 0}}\pi_T\left(\frac{a_i}{d_i}u_i\right) + \sum_{\substack{1 \leq i \leq \ell\\ d_i = 0}}\pi_T( b_iu_i )}{a_i \in \mathbb{Z},\ b_i \in \mathbb{R}}.
	\end{align}
\end{lemma}

Let $E^w$ denote the set of points of $E$ fixed by $w$.
Then
\begin{align}
	E^w = \bigset{x \in E}{w(x) = x} = \bigoplus_{\substack{1 \leq i \leq \ell\\ d_i = 0}}\mathbb{R}u_i
\end{align}
since $\{u_1,\ldots,u_\ell\}$ forms a basis for $E$.
Define a sublattice $M$ of $Z$ by
\begin{align}
	M \ceq \bigoplus_{\substack{1 \leq i \leq \ell\\ d_i = 0}}\mathbb{Z}u_i, \label{M}
\end{align}
and the torus $T_M \ceq E^w/M$.
Then the natural projection from $E^w$ onto $T_M$ is equal to the restriction of $\pi_T$ to $E^w$.
Moreover, define a finite set
\begin{align}
	\XI \ceq \Bigset{\sum_{\substack{1 \leq i \leq \ell\\ d_i > 1}}\frac{a_i}{d_i}u_i \in E}{a_i \in \{1,\ldots,d_i\}}. \label{XI}
\end{align}
For each $\xi \in \XI$, define a map $\epsilon_\xi : \mathbb{Z}_{>0} \lra \{0,1\}$ by
\begin{align}
	\epsilon_\xi(q) = \begin{cases*}
		0 & if $\pi_T(\xi) \not\in T[q]$;\\
		1 & if $\pi_T(\xi) \in T[q]$.
	\end{cases*}
\end{align}
Then we can see that
\begin{align}
	T^w = \bigsqcup_{\xi \in \XI}\bigset{\pi_T(\xi) + t}{t \in T_M},
\end{align}
and 
\begin{align}
	T[q]^w = \bigsqcup_{\substack{\xi \in \XI\\ \epsilon_\xi(q) = 1}}\bigset{\pi_T(\xi) + t}{t \in T_M[q]}
\end{align}
for all $q \in \mathbb{Z}_{>0}$.
Hence we have the following \cite[Equation (3.5)]{Uchiumi}:
\begin{align}
	\chi_{\PHI,q}(w) = \#(T(\A_\PHI) \cap T[q]^w) = \sum_{\xi \in \XI}\epsilon_\xi(q) \cdot \#\bigset{\pi_T(x) \in T_M[q]}{(\beta,\xi) + (\beta,x) \not\in \mathbb{Z} \tforall \beta \in \PHI^+}.
\end{align}
Note that the terms on the right-hand side do not form the characteristic quasi-polynomials of non-central arrangements.

\subsection{Space of fixed points}\ \label{sec4.2}

\noindent
In this subsection, find the integers $d_1,\ldots,d_\ell$ and the basis $\{u_1,\ldots,u_\ell\}$ for each $w \in W_{A_\circ}$.
Recall that $W_{A_\circ} = \OMEGA$ (\cref{prop4.6}).
For $j \in J \setminus \{0\}$, let $\omega \ceq \omega_j \in \OMEGA$ and $\sigma \ceq \sigma_j$.
Let $S_0^j,\ldots,S_r^j$ denote all $\langle\sigma\rangle$-orbits in $I$, that is, $I = S_0^j \sqcup \cdots \sqcup S_r^j$, where we assume that $0 \in S_0^j$.
Let $o(\omega)$ denote the order of $\omega$.
Then we can see that $\#S_0^j = o(\omega)$.
For $k \in \{0,\ldots,r\}$, define 
\begin{align}
	s_k \ceq \min{S_k^j},\qquad m_k \ceq \frac{\#S_k^jn_{s_k}}{\#S_0^j}.
\end{align}
It follows from \cite[Corollary 4.2]{UchiumiRoot} that $m_k \in \mathbb{Z}$ for all $k \in \{0,\ldots,r\}$.

For $x \in E$, define $x^\omega \in E$ by
\begin{align}
	x^\omega \ceq \frac{1}{o(\omega)}\sum_{t=1}^{o(\omega)}\omega^t(x).
\end{align}
It is clear that $x^\omega \in E^\omega$.
For $k \in \{0,\ldots,r\}$, define
\begin{align}
	\pi_k^j \ceq \sum_{s \in S_k^j}\varpivee_s,\qquad
	\bar{\pi}_k^\omega \ceq (\pi_k^j)^\omega = \pi_k^j - m_k\pi_0^j.
\end{align}

\begin{theorem}[{\cite[Theorem 3.4]{UchiumiRoot}}]
	$\{\bar{\pi}_1^\omega,\ldots,\bar{\pi}_r^\omega\}$ is a basis for $E^\omega$.
	Hence $r = r(\omega)$.
\end{theorem}
\begin{theorem}
	The elementary divisors of $(R_\omega - 1)$ with respect to $Z$ is $(1^{\ell-r-1},o(\omega))$, where $1^p$ indicates that $p$ is the number of $1$.
\end{theorem}
\begin{proof}
	Recall that 
	\begin{align}
		\omega(\varpivee_i) = \varpivee_{\sigma(i)} - n_i\varpivee_j
	\end{align}
	for all $i \in I$.
	Then we have
	\begin{align}
		\omega(\pi_0^j) = \pi_0^j - o(\omega)\varpivee_j.
	\end{align}
	For any $k \in \{0,\ldots,r\}$ and $s \in S_k^j\setminus \{s_k\}$, 
	\begin{align}
		\omega(\varpivee_s) = \varpivee_s + (\varpivee_{\sigma(s)} -  \varpivee_s - n_s\varpivee_j).
	\end{align}
	For $i \in \{1,\ldots,\ell\}$, let $u_i,v_i$ and $d_i$ satisfy 
	\begin{align}
		u_i \ceq \begin{cases*}
			\pi_0^j				& if $i = j$;\\
			\bar\pi_k^\omega 	& if $i = s_k$ for some $k \in \{1,\ldots,r\}$;\\
			\varpivee_i			& otherwise,
		\end{cases*}\qquad
		v_i \ceq \begin{cases*}
			-\varpivee_j												& if $i = j$;\\
			\varpivee_{s_k}												& if $i = s_k$ for some $k \in \{1,\ldots,r\}$;\\
			\varpivee_{\sigma(i)} -  \varpivee_i - n_i\varpivee_j		& otherwise,
		\end{cases*}
	\end{align}
	\begin{align}
		d_i \ceq \begin{cases*}
			o(\omega)	& if $i = j$;\\
			0			& if $i = s_k$ for some $k \in \{1,\ldots,r\}$;\\
			1			& otherwise.
		\end{cases*}
	\end{align}
	and define $\ell \times \ell$ matrices $U = (u_{ij})_{ij}$ and $V = (v_{ij})_{ij}$ such that $u_i = \sum_{j=1}^\ell u_{ij}\varpivee_j$ and $v_i = \sum_{j=1}^\ell v_{ij}\varpivee_j$.
	Then $U$ and $V$ are unimodular since $\det{U} = \prod_{i=1}^\ell u_{ii} = 1$ and $\det{V} = \prod_{i=1}^\ell v_{ii} = (-1)^{\ell-r}$.
	Hence we have
	\begin{align}
		U(R_\omega - I)V^{-1} = \diag(d_1,\ldots,d_\ell). \qedend
	\end{align}
\end{proof}

\begin{corollary}
	Let $\omega = \omega_j \in \OMEGA\setminus \{1\}$.
	Then the lattice $M$ defined in \cref{M} and the set $\XI$ defined in \cref{XI} are as follows:
	\begin{align}
		M = \bigoplus_{k=1}^r\mathbb{Z}\bar\pi_k^\omega,\qquad
		\XI = \Bigset{\frac{a}{o(\omega)}\pi_0^j}{a \in \{1,\ldots,o(\omega)\}}.
	\end{align}
\end{corollary}

Therefore, for $\xi = \frac{a}{o(\omega)}\pi_0^j \in \XI$ and $q \in \mathbb{Z}_{>0}$, the following are equivalent:
\begin{iitemize}
	\item $\epsilon_\xi(q) = 1$;
	\item $\pi_T(\xi) \in T[q]$;
	\item $a \in \frac{o(\omega)}{\gcd\{o(\omega),q\}}\mathbb{Z}$.
\end{iitemize}

\subsection{Via root systems obtianed by folding}\ \label{secnew4.3}

\noindent
Let $\omega = \omega_j \in \OMEGA\setminus \{1\}$.
Recall that 
\begin{align}
	\chi_{\PHI,q}(\omega) &= \sum_{\xi \in \XI}\epsilon_\xi(q) \cdot \#\bigset{\pi_T(x) \in T_M[q]}{(\beta,\xi) + (\beta,x) \not\in \mathbb{Z} \tforall \beta \in \PHI^+}\\
	&= \sum_{\xi \in \XI}\epsilon_\xi(q) \cdot \#\bigset{\pi_T(x) \in T_M[q]}{(\beta,\xi) + (\beta,x) \not\in \mathbb{Z} \tforall \beta \in \PHI}.
\end{align}
Define sets $\PHI^\omega$, $\PHI^\omega_\re$ and $\DELTA^\omega$ as
\begin{align}
	\PHI^\omega \ceq \bigset{\beta^\omega \in E^\omega}{\beta \in \PHI},\qquad
	\PHI^\omega_\re \ceq \bigset{\beta^\omega \in E^\omega}{\beta \in \PHI,\ \beta^\omega \neq 0}\qquad
	\DELTA^\omega \ceq \bigset{\alpha_i^\omega \in \PHI^\omega}{i \in I \setminus S_0^j}.
\end{align}

\begin{theorem}[{\cite[Theorem 3.7, Theorem 4.1]{UchiumiRoot}}]
	Let $\PHI$ be an irreducible reduced root system.
	Then $\PHI^\omega_\re$ is an irreducible root system.
	Furthermore, if $r(\omega) > 0$ (i.e., $\PHI^\omega_\re \neq \emptyset$), then $\DELTA^\omega$ is a basis of $\PHI_\re^\omega$, and the lattice $M$ is the coweight lattice of $\PHI^\omega_\re$.
\end{theorem}
The root system $\PHI^\omega_\re$ can be characterized by the folding of the extended Dynkin diagram \cite[\S4.2]{UchiumiRoot}.

If $x \in E^\omega$, then 
\begin{align}
	(\omega(\beta),x) = (\omega(\beta,\omega(x))) = (\beta,x)
\end{align}
for all $\beta \in \PHI$.
Therefore, for any $\beta \in \PHI$ and $x \in E^\omega$, we have
\begin{align}
	(\beta^\omega,x) = (\beta,x).
\end{align}

For any $\beta^\omega \in \PHI^\omega$, define a set $\PI(\beta^\omega)$ by 
\begin{align}
	P(\beta^\omega) \ceq \bigset{(\gamma,\, \pi_0^j)}{\gamma \in \PHI,\ \gamma^\omega = \beta^\omega}.
\end{align}
For $q \in \mathbb{Z}_{>0}$, let $o(\omega)_q \ceq \frac{o(\omega)}{\gcd\{o(\omega),q\}}$, then we can see that
\begin{align}
	\chi_{\PHI,q}(\omega) 
	&= \sum_{\xi \in \XI}\epsilon_\xi(q) \cdot \#\bigset{\pi_T(x) \in T_M[q]}{(\beta,\xi) + (\beta,x) \not\in \mathbb{Z} \tforall \beta \in \PHI}\\
	&= \sum_{\substack{1 \leq a \leq o(\omega)\\a \in o(\omega)_q \mathbb{Z}}}\#\Bigset{\pi_T(x) \in T_M[q]}{\frac{ap}{o(\omega)} + (\beta^\omega,x) \not\in \mathbb{Z} \tforall \beta^\omega \in \PHI^\omega,\ p \in P(\beta^\omega)}. \label{chii}
\end{align}
Suppose that $\beta^\omega = 0$, then $P(\beta^\omega) = P(0) = \{\pm1,\,\ldots,\, \pm(o(\omega)-1)\}$ \cite[Proposition 5.24]{UchiumiRoot}.
If $\frac{ap}{o(\omega)} \in \mathbb{Z}$ for some $p \in P(0)$, then the term of \cref{chii} corresponding to $a$ is equal to $0$.
In other words, for the term do not equal $0$, it must be that $\gcd\{o(\omega),a\} = 1$.
Furthermore, if $a \in o(\omega)_q\mathbb{Z}$ satisfies $\gcd\{o(\omega),a\} = 1$, then $q$ must be multiple of $o(\omega)$.
Define a periodic function $c : \mathbb{Z}_{>0} \lra \mathbb{Z}$ by 
\begin{align}
	c(q) = \begin{cases*}
		0					& if $q \not\in o(\omega)\mathbb{Z}$;\\
		\varphi(o(\omega))	& if $q \in o(\omega)\mathbb{Z}$,
	\end{cases*}
\end{align}
where $\varphi$ is the Euler's totient function, that is, $\varphi(n)$ is equal to be the number of positive integers less than or equal to $n$ which are relatively prime to $n$.
Then we have 
\begin{align}
	\chi_{\PHI,q}(\omega) = c(q) \cdot \#\Bigset{\pi_T(x) \in T_M[q]}{\frac{p}{o(\omega)} + (\beta^\omega,x) \not\in \mathbb{Z} \tforall \beta^\omega \in \PHI_\re^\omega,\ p \in P(\beta^\omega)}. \label{chiphiq}
\end{align}

The root system $\PHI_\re^\omega$ and the set $P(\beta^\omega)$ are as shown in \cref{tableE}, and based on this, one of the main results of this paper is obtained as follows:

\begin{table}[h]
	
	\caption{List of type of $\PHI_\re^\omega$ and $P(\beta^\omega)$}\label{tableE}
	
	\centering

	\begin{tabular}{cc|c|c|Sc}
		\multicolumn{2}{c|}{$\PHI$}                         & $j$                              & $\PHI^{\omega}_\re$ 			& $P(\beta^\omega)$  \\ \hline\hline
		$A_\ell$  &                         & $\gcd\{\ell+1,j\} = 1$		   & $\varnothing$         			& --- \\ 
		$A_\ell$  &                         & $g \ceq \gcd\{\ell+1,j\} \neq 1$ & $A_{g-1}$         				& $\{0,\, \pm1,\, \ldots,\, \pm(o(\omega)-1)\}$ \\ 
		$B_2$     &                         & $j = 1$                          & $A_1$               		   	& $\{0,\pm1\}$ \\ 
		$B_\ell$  & $(\ell \geq 3)$			& $j = 1$                          & $B_{\ell-1}$          		   	& $\begin{cases*}
			\{0,\pm1\}		& if $\beta^\omega$ is a short root of $\PHI_\re^\omega$;\\
			\{0\}			& if $\beta^\omega$ is a long root of $\PHI_\re^\omega$
		\end{cases*}$ \\ 
		$C_\ell$  & $(\ell \geq 3$, odd)	& $j = \ell$                       & $BC_{\frac{\ell-1}{2}_{_{}}}$ 	& $\{0,\pm1\}$ \\ 
		$C_\ell$  & $(\ell \geq 4$, even) 	& $j = \ell$                       & $C_{\frac{\ell}{2}_{_{}}}$    	& $\{0,\pm1\}$ \\ 
		$D_\ell$  & $(\ell \geq 4$)   		& $j = 1 $                         & $B_{\ell-2}$    				& $\begin{cases*}
			\{0,\pm1\}		& if $\beta^\omega$ is a short root of $\PHI_\re^\omega$;\\
			\{0\}			& if $\beta^\omega$ is a long root of $\PHI_\re^\omega$
		\end{cases*}$ \\ 
		$D_\ell$  & $(\ell \geq 4$, even) 	& $j \in \{\ell-1,\, \ell\}$       & $C_{\frac{\ell}{2}_{_{}}}$    	& $\begin{cases*}
			\{0,\pm1\}		& if $\beta^\omega$ is a short root of $\PHI_\re^\omega$;\\
			\{0\}			& if $\beta^\omega$ is a long root of $\PHI_\re^\omega$
		\end{cases*}$  \\ 
		$D_\ell$  & $(\ell \geq 5$, odd)  	& $j \in \{\ell-1,\, \ell\}$       & $BC_{\frac{\ell-3}{2}_{_{}}}$  & $\begin{cases*}
			\{0,\pm1,\pm2,\pm3\}		& if $\beta^\omega$ is a short root of $\PHI_\re^\omega$;\\
			\{0,\pm2\}			& if $\beta^\omega$ is not a short root of $\PHI_\re^\omega$
		\end{cases*}$ \\ 
		$E_6$     &                         & $j \in \{1,6\}$                  & $G_2$          			    & $\begin{cases*}
			\{0,\pm1,\pm2\}		& if $\beta^\omega$ is a short root of $\PHI_\re^\omega$;\\
			\{0\}			& if $\beta^\omega$ is a long root of $\PHI_\re^\omega$
		\end{cases*}$ \\ 
		$E_7$     &                         & $j = 7$                          & $F_4$               			& $\begin{cases*}
			\{0,\pm1\}		& if $\beta^\omega$ is a short root of $\PHI_\re^\omega$;\\
			\{0\}			& if $\beta^\omega$ is a long root of $\PHI_\re^\omega$
		\end{cases*}$ \\ 
	\end{tabular}
	
	\

\end{table}

\begin{theorem}\label{thm5.6}
	Let $\PHI$ be an arbitrary irreducible reduced root system, and $\omega = \omega_j \in W_{A_\circ}$.
	Let $\PHI'$ be a root system and $d \in \mathbb{Z}_{>0}$, as given in \cref{tablePHId}.
	Then we have
	\begin{align}
		\chi_{\PHI,q}(\omega) = c(q) \cdot \chi_{d\PHI',M}^\quasi(q).
	\end{align}
	Recall that $M$ is the coweight lattice of $\PHI_\re^\omega$.
\end{theorem}

\begin{table}[h]
	
	\caption{List of type of $\PHI_\re^\omega$ and $P(\beta^\omega)$}\label{tablePHId}
	
	\centering

	\begin{tabular}{cc|c|c|c|c}
		\multicolumn{2}{c|}{$\PHI$}                        & $j$                              & $\PHI^{\omega}_\re$ 			& $\PHI'$ & $d$  \\ \hline\hline
		$A_\ell$  &                         & $\gcd\{\ell+1,j\} = 1$		   & $\varnothing$         			& $\varnothing$ & $1$ \\ 
		$A_\ell$  &                         & $g \ceq \gcd\{\ell+1,j\} \neq 1$ & $A_{g-1}$         				& $A_{g-1}$ & $o(\omega_j) = \frac{\ell+1}{g}$ \\ 
		$B_2$     &                         & $j = 1$                          & $A_1$               		   	& $A_1$ & $2$ \\ 
		$B_\ell$  & $(\ell \geq 3)$			& $j = 1$                          & $B_{\ell-1}$          		   	& $C_{\ell-1}$ & $1$ \\ 
		$C_\ell$  & $(\ell \geq 3$, odd)	& $j = \ell$                       & $BC_{\frac{\ell-1}{2}_{_{}}}$ 	& $BC_{\frac{\ell-1}{2}}$ & $2$ \\ 
		$C_\ell$  & $(\ell \geq 4$, even) 	& $j = \ell$                       & $C_{\frac{\ell}{2}_{_{}}}$    	&  $C_{\frac{\ell}{2}}$ & $2$ \\ 
		$D_\ell$  & $(\ell \geq 4$)   		& $j = 1 $                         & $B_{\ell-2}$    				& $C_{\ell-2}$ &  $1$ \\ 
		$D_\ell$  & $(\ell \geq 4$, even) 	& $j \in \{\ell-1,\, \ell\}$       & $C_{\frac{\ell}{2}_{_{}}}$    	& $B_{\frac{\ell}{2}}$ & $2$  \\ 
		$D_\ell$  & $(\ell \geq 5$, odd)  	& $j \in \{\ell-1,\, \ell\}$       & $BC_{\frac{\ell-3}{2}_{_{}}}$  & $C_{\frac{\ell-3}{2}}$ & $2$ \\ 
		$E_6$     &                         & $j \in \{1,6\}$                  & $G_2$          			    & $(\PHI^\omega_\re)^\vee$ & $1$ \\ 
		$E_7$     &                         & $j = 7$                          & $F_4$               			& $(\PHI^\omega_\re)^\vee$ & $1$ \\ 
	\end{tabular}
	
	\

\end{table}

We will leave the details of how to determine $\PHI'$ and $d$, and the explicit calculations to the next section.

\begin{corollary}\label{chichi}
	Let $\PHI$ be an arbitrary irreducible reduced root system with the exponents $e_1,\ldots,e_\ell$ and the index of connection $f$.
	If $\gcd\{f,q\} = 1$, then 
	\begin{align}
		\chi_{\PHI,q} = \frac{\chi_{\PHI}^\quasi(q)}{\#W} \cdot \chi_{\st},
	\end{align}
	where $\chi_{\st}$ is the regular character of $W$:
	\begin{align}
		\chi_{\st}(w) = \begin{cases*}
			\#W	& if $w = 1$;\\
			0	& if $w \neq 1$.
		\end{cases*}
	\end{align}
	Furthermore, if $\gcd\{f,\tilde{n}_\PHI,q\} = 1$, then 
	\begin{align}
		\chi_{\PHI,q} = \frac{\chi_{\A_\PHI}(q)}{\#W} \cdot \chi_{\st} = \frac{(q-e_1) \cdots (q-e_\ell)}{\#W} \cdot \chi_\st.
	\end{align}
\end{corollary}
\begin{proof}
	For the identity element $1 \in W$, we have $\chi_{\PHI,q}(1) = \chi_\PHI^\quasi(q)$.
	In particular, if $\gcd\{\tilde{n}_\PHI,q\} = 1$, then $\chi_{\PHI,q}(1) = \chi_{\A_\PHI}(q)$ by \cref{KamiyaTakemuraTerao}.
	
	Let $q \in \mathbb{Z}_{>0}$ satisfy $\gcd\{f,q\} = 1$.
	Since $f = \#Z/\veeQ = \#W_{A_\circ}$, $q$ is relatively prime to $o(w)$ for all $w \in W_{A_\circ} \setminus\{1\}$.
	Therefore $\chi_{\PHI,q}(w) = 0$ for all $w \in W_{A_\circ}\setminus\{1\}$, and hence
	\begin{align}
		\chi_{\PHI,q}(w) = \begin{cases*}
			\chi_\PHI^\quasi(q)	& if $w = 1$;\\
			0	& if $w \neq 1$.
		\end{cases*}\qedend
	\end{align}
\end{proof}

\begin{remark}
	In general, the statement such as \cref{chichi} does not hold for hyperplane arrangements (See \cite[Example 2.11]{Uchiumi}).
\end{remark}

\section{Details of the calculation} \label{sec5}

First, we will discuss the following elementary number theoretic lemma.

\begin{lemma}\label{lem6.1}
	Let $d \in \mathbb{Z}_{>0}$.
	For $x \in \mathbb{R}$, the following are equivalent:
	\begin{eenumeratei}
		\item $dx \not\in \mathbb{Z}$;
		\item $x + \frac{p}{d} \not\in \mathbb{Z}$ for all $p \in \{0,\pm1\ldots,\pm(d-1)\}$.
	\end{eenumeratei}
\end{lemma}
\begin{proof}
	Suppose that $dx \in \mathbb{Z}$ and let $c \ceq \lceil x \rceil -x$, where $\lceil x \rceil = \min\tbigset{n \in \mathbb{Z}}{x \leq n}$.
	Then $0 \leq c < 1$, and $dc = d \lceil x \rceil - dx$ is an integer.
	Hence $c = \frac{p}{d}$ for some $p \in \{0,\ldots,d-1\}$, and we have $x + \frac{p}{d} = \lceil x \rceil \in \mathbb{Z}$.
	
	Conversely, suppose that $x + \frac{p}{d} \in \mathbb{Z}$ for some $p \in \{0,\pm1\ldots,\pm(d-1)\}$.
	Then $dx + p$ is an integer, and hence $dx \in \mathbb{Z}$.
\end{proof}

\subsection{type $A_\ell$}\ 

\noindent 
Let $\PHI$ be a root system of type $A_\ell$, and $\omega = \omega_j \in W_{A_\circ} \setminus \{1\}$.
If $\gcd\{\ell+1,\, j\} = 1$, then 
\begin{align}
	\chi_{\PHI,q}(\omega) = c(q) = c(q) \cdot \chi_{\varnothing}^\quasi(q).
\end{align}
 by \cref{chiphiq}.
Suppose that $g \ceq \gcd\{\ell+1,\, j\} \neq 1$ and let $d \ceq o(\omega) = \frac{\ell+1}{g}$.
Then $\PHI^\omega_\re$ is of type $A_{g-1}$ and $P(\beta^\omega) = \{0,\pm1,\ldots,\pm(d-1)\}$ for all $\beta^\omega \in \PHI_\re^\omega$.
Therefore, by \cref{lem6.1}, the equation \cref{chiphiq} can be expressed as follows:
\begin{align}
	\chi_{\PHI,q}(\omega) &= c(q) \cdot \#\Bigset{\pi_T(x) \in T_M[q]}{\frac{p}{d} + (\beta^\omega,x) \not\in \mathbb{Z} \tforall \beta^\omega \in \PHI_\re^\omega,\ p \in P(\beta^\omega)}\\
	&= c(q) \cdot \#\bigset{\pi_T(x) \in T_M[q]}{(d\beta^\omega,x) \not\in \mathbb{Z} \tforall \beta^\omega \in \PHI_\re^\omega}\\
	&= c(q) \cdot \chi_{d\PHI_\re^\omega,M}^\quasi(q).
\end{align}

\begin{theorem}[type $A_\ell$]\label{typeAell}
	Let $\PHI$ be a root system of type $A_\ell$, and let $\omega = \omega_j \in W_{A_\circ}$, $g \ceq \gcd\{\ell+1,\, j\}$ and $d \ceq o(\omega)$.
	Then 
	\begin{align}
		\chi_{\PHI,q}(\omega) = \begin{cases*}
			\!\!\!\begin{array}{ll}
			0										& \textup{if $q \not\in d\mathbb{Z}$;}\\
			\varphi(d)(q-d)(q-2d) \cdots (q-d(g-1))	& \textup{if $q \in d\mathbb{Z}$}.
			\end{array}
		\end{cases*}
	\end{align}
\end{theorem}
\begin{proof}
	Use \cref{dilated} for the formula of \cref{typeAroot}.
\end{proof}

\subsection{type $B_\ell$}\ 

\noindent
Let $\PHI$ be a root system of type $B_\ell$, and $\omega = \omega_j \in W_{A_\circ} \setminus \{1\}$ (i.e., $j = 1$).
Then $o(\omega) = 2$.

If $\ell= 2$, then $\PHI_\re^\omega$ is of type $A_1$ and $P(\beta^\omega) = \{0,\pm1\}$.
Therefore, by \cref{lem6.1}, the equation \cref{chiphiq} can be expressed as follows:
\begin{align}
	\chi_{\PHI,q}(\omega) &= c(q) \cdot \#\Bigset{\pi_T(x) \in T_M[q]}{\frac{p}{2} + (\beta^\omega,x) \not\in \mathbb{Z} \tforall \beta^\omega \in \PHI_\re^\omega,\ p \in P(\beta^\omega)}\\
	&= c(q) \cdot \#\bigset{\pi_T(x) \in T_M[q]}{(2\beta^\omega,x) \not\in \mathbb{Z} \tforall \beta^\omega \in \PHI_\re^\omega}\\
	&= c(q) \cdot \chi_{2\PHI_\re^\omega,M}^\quasi(q).
\end{align}

If $\ell \geq 3$, then $\PHI_\re^\omega$ is of type $B_{\ell-1}$, and 
\begin{align}
	P(\beta^\omega) = \begin{cases*}
		\{0,\pm1\}		& if $\beta^\omega$ is a short root of $\PHI_\re^\omega$;\\
		\{0\}			& if $\beta^\omega$ is a long root of $\PHI_\re^\omega$.
	\end{cases*}
\end{align}
Therefore, by \cref{lem6.1}, the equation \cref{chiphiq} can be expressed as follows:
\begin{align}
	\chi_{\PHI,q}(\omega) &= c(q) \cdot \#\Bigset{\pi_T(x) \in T_M[q]}{\frac{p}{2} + (\beta^\omega,x) \not\in \mathbb{Z} \tforall \beta^\omega \in \PHI_\re^\omega,\ p \in P(\beta^\omega)}\\
	&= c(q) \cdot \#\Bigset{\pi_T(x) \in T_M[q]}{
		\begin{lgathered}
			(2\beta^\omega,x) \not\in \mathbb{Z} \text{ if $\beta^\omega \in \PHI_\re^\omega$ is a short root},\\
			(\beta^\omega,x) \not\in \mathbb{Z} \text{ if $\beta^\omega \in \PHI_\re^\omega$ is a long root}
		\end{lgathered}}
\end{align}
Let $\PHI'$ be a root system of type $C_{\ell-1}$.
Then we have
\begin{align}
	\chi_{\PHI,q}(\omega) 
	&= c(q) \cdot \#\Bigset{\pi_T(x) \in T_M[q]}{
		\begin{lgathered}
			(2\beta^\omega,x) \not\in \mathbb{Z} \text{ if $\beta^\omega \in \PHI_\re^\omega$ is a short root},\\
			(\beta^\omega,x) \not\in \mathbb{Z} \text{ if $\beta^\omega \in \PHI_\re^\omega$ is a long root}
	\end{lgathered}}\\
	&= c(q) \cdot \#\bigset{\pi_T(x) \in T_M[q]}{(\beta,x) \not\in \mathbb{Z} \tforall \beta \in \PHI'}\\
	&= c(q) \cdot \chi_{\PHI',M}^\quasi(q).
\end{align}

\begin{theorem}[type $B_\ell$]\label{typeBell}
	Let $\PHI$ be a root system of type $B_\ell$, and $\omega = \omega_j \in W_{A_\circ}$.
	Then 
	\begin{align}
		\chi_{\PHI,q}(\omega) = 		
		\begin{cases*}
			\!\!\!\begin{array}{ll}
				(q-1)(q-3) \cdots (q-(2\ell-1))				& \textup{if $\omega = 1$ and $q \not\in 2\mathbb{Z}$;}\\
				(q-2)(q-4) \cdots (q-2(\ell-1))(q-\ell)		& \textup{if $\omega = 1$ and $q \in 2\mathbb{Z}$;}\\
			\end{array}\\
			\!\!\!\begin{array}{ll}
				0											& \textup{if $\omega = \omega_1$ and $q \not\in 2\mathbb{Z}$;}\\
				(q-2)(q-4) \cdots (q-2(\ell-1))				& \textup{if $\omega = \omega_1$ and $q \in 2\mathbb{Z}$}.
			\end{array}
		\end{cases*}
	\end{align}
\end{theorem}
\begin{proof}
	In the case where $\omega = 1$, see \cref{typeBroot}.
	If $\omega = \omega_1$, then $\chi_{\PHI,q}(\omega)$ is equal to the second formula of \cref{typeCroot}.
\end{proof}

\subsection{type $C_\ell$}\ 

\noindent
Let $\PHI$ be a root system of type $C_\ell$, and $\omega = \omega_j \in W_{A_\circ} \setminus \{1\}$ (i.e., $j = \ell$).
Then $o(\omega) = 2$.

If $\ell$ is odd, then $\PHI_\re^\omega$ is of type $BC_{\frac{\ell-1}{2}}$, and $P(\beta^\omega) = \{0,\pm1\}$ for all $\beta^\omega \in \PHI_\re^\omega$.
Therefore, by \cref{lem6.1}, the equation \cref{chiphiq} can be expressed as follows:
\begin{align}
	\chi_{\PHI,q}(\omega) = &= c(q) \cdot \#\Bigset{\pi_T(x) \in T_M[q]}{\frac{p}{2} + (\beta^\omega,x) \not\in \mathbb{Z} \tforall \beta^\omega \in \PHI_\re^\omega,\ p \in P(\beta^\omega)}\\
	&= c(q) \cdot \#\bigset{\pi_T(x) \in T_M[q]}{(2\beta^\omega,x) \not\in \mathbb{Z} \tforall \beta^\omega \in \PHI_\re^\omega}\\
	&= c(q) \cdot \chi_{2\PHI_\re^\omega,M}^\quasi(q).
\end{align}

\begin{theorem}
	Let $\PHI$ be a root system of type $C_\ell$, $\ell \not\in 2\mathbb{Z}$, and $\omega = \omega_j \in W_{A_\circ}$.
	Then 
	\begin{align}
		\chi_{\PHI,q}(\omega) = 
		\begin{cases*}
			\!\!\!\begin{array}{ll}
				(q-1)(q-3) \cdots (q-(2\ell-1))				& \textup{if $\omega = 1$ and $q \not\in 2\mathbb{Z}$;}\\
				(q-2)(q-4) \cdots (q-2(\ell-1))(q-\ell)		& \textup{if $\omega = 1$ and $q \in 2\mathbb{Z}$;}\\
			\end{array}\\
			\!\!\!\begin{array}{ll}
				0											& \textup{if $\omega = \omega_{\ell}$ and $\gcd\{4,q\} = 1$;}\\
				(q-2)(q-6) \cdots (q-2(\ell-2))				& \textup{if $\omega = \omega_{\ell}$ and $\gcd\{4,q\} = 2$;}\\
				(q-4)(q-8) \cdots (q-2(\ell-1))				& \textup{if $\omega = \omega_{\ell}$ and $\gcd\{4,q\} = 4$}.
			\end{array}
		\end{cases*}
	\end{align}
\end{theorem}
\begin{proof}
	In the case where $\omega = 1$, see \cref{typeCroot}.
	When $\omega = \omega_\ell$, use \cref{dilated} for the formula of \cref{typeBCroot}.
\end{proof}

If $\ell$ is even, then $\PHI_\re^\omega$ is of type $C_{\frac{\ell}{2}}$, and $P(\beta^\omega) = \{0,\pm1\}$ for all $\beta^\omega \in \PHI_\re^\omega$.
Therefore, by \cref{lem6.1}, the equation \cref{chiphiq} can be expressed as follows:
\begin{align}
	\chi_{\PHI,q}(\omega) = &= c(q) \cdot \#\Bigset{\pi_T(x) \in T_M[q]}{\frac{p}{2} + (\beta^\omega,x) \not\in \mathbb{Z} \tforall \beta^\omega \in \PHI_\re^\omega,\ p \in P(\beta^\omega)}\\
	&= c(q) \cdot \#\bigset{\pi_T(x) \in T_M[q]}{(2\beta^\omega,x) \not\in \mathbb{Z} \tforall \beta^\omega \in \PHI_\re^\omega}\\
	&= c(q) \cdot \chi_{2\PHI_\re^\omega,M}^\quasi(q).
\end{align}

\begin{theorem}
	Let $\PHI$ be a root system of type $C_\ell$, $\ell \in 2\mathbb{Z}$, and $\omega = \omega_j \in W_{A_\circ}$.
	Then 
	\begin{align}
		\chi_{\PHI,q}(\omega) = \begin{cases*}
			\!\!\!\begin{array}{ll}
			(q-1)(q-3) \cdots (q-(2\ell-1))				& \textup{if $\omega = 1$ and $q \not\in 2\mathbb{Z}$;}\\
			(q-2)(q-4) \cdots (q-2(\ell-1))(q-\ell)		& \textup{if $\omega = 1$ and $q \in 2\mathbb{Z}$;}\\
			\end{array}\\
			\!\!\!\begin{array}{ll}
			0											& \textup{if $\omega = \omega_{\ell}$ and $\gcd\{4,q\} = 1$;}\\
			(q-2)(q-6) \cdots (q-2(\ell-1))				& \textup{if $\omega = \omega_{\ell}$ and $\gcd\{4,q\} = 2$;}\\
			(q-4)(q-8) \cdots (q-2(\ell-2))(q-\ell)		& \textup{if $\omega = \omega_{\ell}$ and $\gcd\{4,q\} = 4$}.
			\end{array}
		\end{cases*}
	\end{align}
\end{theorem}
\begin{proof}
	In the case where $\omega = 1$, see \cref{typeCroot}.
	When $\omega = \omega_\ell$, use \cref{dilated} for the first formula of \cref{typeCroot}.
\end{proof}

\subsection{type $D_\ell$}\ 

\noindent
Let $\PHI$ be a root system of type $D_\ell$, and $\omega = \omega_j \in W_{A_\circ} \setminus \{1\}$.

If $j = 1$, then $o(\omega) = 2$, $\PHI_\re^\omega$ is of type $B_{\ell-2}$, and 
\begin{align}
	P(\beta^\omega) = \begin{cases*}
		\{0,\pm1\}		& if $\beta^\omega$ is a short root of $\PHI_\re^\omega$;\\
		\{0\}			& if $\beta^\omega$ is a long root of $\PHI_\re^\omega$.
	\end{cases*}
\end{align}
Therefore, by \cref{lem6.1}, the equation \cref{chiphiq} can be expressed as follows:
\begin{align}
	\chi_{\PHI,q}(\omega) &= c(q) \cdot \#\Bigset{\pi_T(x) \in T_M[q]}{\frac{p}{2} + (\beta^\omega,x) \not\in \mathbb{Z} \tforall \beta^\omega \in \PHI_\re^\omega,\ p \in P(\beta^\omega)}\\
	&= c(q) \cdot \#\Bigset{\pi_T(x) \in T_M[q]}{
		\begin{lgathered}
			(2\beta^\omega,x) \not\in \mathbb{Z} \text{ if $\beta^\omega \in \PHI_\re^\omega$ is a short root},\\
			(\beta^\omega,x) \not\in \mathbb{Z} \text{ if $\beta^\omega \in \PHI_\re^\omega$ is a long root}
	\end{lgathered}}
\end{align}
Let $\PHI'$ be a root system of type $C_{\ell-2}$.
Then we have
\begin{align}
	\chi_{\PHI,q}(\omega) 
	&= c(q) \cdot \#\Bigset{\pi_T(x) \in T_M[q]}{
		\begin{lgathered}
			(2\beta^\omega,x) \not\in \mathbb{Z} \text{ if $\beta^\omega \in \PHI_\re^\omega$ is a short root},\\
			(\beta^\omega,x) \not\in \mathbb{Z} \text{ if $\beta^\omega \in \PHI_\re^\omega$ is a long root}
	\end{lgathered}}\\
	&= c(q) \cdot \#\bigset{\pi_T(x) \in T_M[q]}{(\beta,x) \not\in \mathbb{Z} \tforall \beta \in \PHI'}\\
	&= c(q) \cdot \chi_{\PHI',M}^\quasi(q).
\end{align}

Suppose that $j \in \{\ell-1,\, \ell\}$.
If $\ell$ is even, then $\PHI_\re^\omega$ is of type $C_{\frac{\ell}{2}}$, and 
\begin{align}
	P(\beta^\omega) = \begin{cases*}
		\{0,\pm1\}		& if $\beta^\omega$ is a short root of $\PHI_\re^\omega$;\\
		\{0\}			& if $\beta^\omega$ is a long root of $\PHI_\re^\omega$.
	\end{cases*}
\end{align}
Therefore, by \cref{lem6.1}, the equation \cref{chiphiq} can be expressed as follows:
\begin{align}
	\chi_{\PHI,q}(\omega) &= c(q) \cdot \#\Bigset{\pi_T(x) \in T_M[q]}{\frac{p}{2} + (\beta^\omega,x) \not\in \mathbb{Z} \tforall \beta^\omega \in \PHI_\re^\omega,\ p \in P(\beta^\omega)}\\
	&= c(q) \cdot \#\Bigset{\pi_T(x) \in T_M[q]}{
		\begin{lgathered}
			(2\beta^\omega,x) \not\in \mathbb{Z} \text{ if $\beta^\omega \in \PHI_\re^\omega$ is a short root},\\
			(\beta^\omega,x) \not\in \mathbb{Z} \text{ if $\beta^\omega \in \PHI_\re^\omega$ is a long root}
	\end{lgathered}}
\end{align}
Let $\PHI'$ be a root system of type $B_{\frac{\ell}{2}}$.
Then we have
\begin{align}
	\chi_{\PHI,q}(\omega) 
	&= c(q) \cdot \#\Bigset{\pi_T(x) \in T_M[q]}{
		\begin{lgathered}
			(2\beta^\omega,x) \not\in \mathbb{Z} \text{ if $\beta^\omega \in \PHI_\re^\omega$ is a short root},\\
			(\beta^\omega,x) \not\in \mathbb{Z} \text{ if $\beta^\omega \in \PHI_\re^\omega$ is a long root}
	\end{lgathered}}\\
	&= c(q) \cdot \#\bigset{\pi_T(x) \in T_M[q]}{(2\beta,x) \not\in \mathbb{Z} \tforall \beta \in \PHI'}\\
	&= c(q) \cdot \chi_{2\PHI',M}^\quasi(q).
\end{align}
If $\ell$ is odd, then $\PHI_\re^\omega$ is of type $BC_{\frac{\ell-3}{2}}$, and 
\begin{align}
	P(\beta^\omega) = \begin{cases*}
		\{0,\pm1,\pm2,\pm3\}		& if $\beta^\omega$ is a short root of $\PHI_\re^\omega$;\\
		\{0,\pm2\}			& if $\beta^\omega$ is not a short root of $\PHI_\re^\omega$
	\end{cases*}
\end{align}
Therefore, by \cref{lem6.1}, the equation \cref{chiphiq} can be expressed as follows:
\begin{align}
	\chi_{\PHI,q}(\omega) &= c(q) \cdot \#\Bigset{\pi_T(x) \in T_M[q]}{\frac{p}{2} + (\beta^\omega,x) \not\in \mathbb{Z} \tforall \beta^\omega \in \PHI_\re^\omega,\ p \in P(\beta^\omega)}\\
	&= c(q) \cdot \#\Bigset{\pi_T(x) \in T_M[q]}{
		\begin{lgathered}
			(4\beta^\omega,x) \not\in \mathbb{Z} \text{ if $\beta^\omega \in \PHI_\re^\omega$ is a short root},\\
			(2\beta^\omega,x) \not\in \mathbb{Z} \text{ if $\beta^\omega \in \PHI_\re^\omega$ is not a short root}
	\end{lgathered}}
\end{align}
Let $\PHI'$ be a root system of type $C_{\frac{\ell-3}{2}}$.
Then we have
\begin{align}
	\chi_{\PHI,q}(\omega) 
	&= c(q) \cdot \#\Bigset{\pi_T(x) \in T_M[q]}{
		\begin{lgathered}
			(4\beta^\omega,x) \not\in \mathbb{Z} \text{ if $\beta^\omega \in \PHI_\re^\omega$ is a short root},\\
			(2\beta^\omega,x) \not\in \mathbb{Z} \text{ if $\beta^\omega \in \PHI_\re^\omega$ is not a short root}
		\end{lgathered}}\\
	&= c(q) \cdot \#\bigset{\pi_T(x) \in T_M[q]}{(2\beta,x) \not\in \mathbb{Z} \tforall \beta \in \PHI'}\\
	&= c(q) \cdot \chi_{2\PHI',M}^\quasi(q).
\end{align}

\begin{theorem}[type $D_\ell$ ($\ell$ is even)]
	Let $\PHI$ be a root system of type $D_\ell$, $\ell \in 2\mathbb{Z}$, and $\omega = \omega_j \in W_{A_\circ}$.
	Then 
	\begin{align}
		\chi_{\PHI,q}(\omega) = \begin{cases*}
			\!\!\!\begin{array}{ll}
			(q-1)(q-3) \cdots (q-(2\ell-3))(q-(\ell-1))										
					& \textup{if $\omega = 1$ and $q \not\in 2\mathbb{Z}$;}\\
			(q-2)(q-4) \cdots (q-2(\ell-2))(q^2 -2(\ell-1)q + \frac{\ell(\ell-1)}{2})		
					& \textup{if $\omega = 1$ and $q \in 2\mathbb{Z}$;}\\
			\end{array}\\
			\!\!\!\begin{array}{ll}
			0
					& \textup{if $\omega = \omega_1$ and $q \not\in 2\mathbb{Z}$;}\\
			(q-2)(q-4) \cdots (q-2(\ell-2))													
					& \textup{if $\omega = \omega_1$ and $q \in 2\mathbb{Z}$;}\\
			\end{array}\\
			\!\!\!\begin{array}{ll}
			0
					& \textup{if $\omega \in \{\omega_{\ell-1},\,\omega_\ell\}$ and $\gcd\{4,q\} = 1$;}\\
			(q-2)(q-6) \cdots (q-2(\ell-3))(q - (\frac{3\ell}{2}-2))				
					& \textup{if $\omega \in \{\omega_{\ell-1},\,\omega_\ell\}$ and $\gcd\{4,q\} = 2$;}\\
			(q-4)(q-8) \cdots (q-2(\ell-2))(q-\frac{\ell}{2})		
					& \textup{if $\omega \in \{\omega_{\ell-1},\,\omega_\ell\}$ and $\gcd\{4,q\} = 4$}.
			\end{array}
		\end{cases*}
	\end{align}
\end{theorem}
\begin{proof}
	In the case where $\omega = 1$, see \cref{typeDroot}.
	If $\omega =\omega_1$, then $\chi_{\PHI,q}(\omega)$ is equal to the second formula of \cref{typeCroot}.
	When $\omega \in \{\omega_{\ell-1},\omega_\ell\}$, use \cref{dilated} for the second formula of \cref{typeCroot}.
\end{proof}

\begin{theorem}[type $D_\ell$ ($\ell$ is odd)]
	Let $\PHI$ be a root system of type $D_\ell$, $\ell \not\in 2\mathbb{Z}$, and $\omega = \omega_j \in W_{A_\circ}$.
	Then 
	\begin{align}
		\chi_{\PHI,q}(\omega) = \begin{cases*}
			\!\!\!\begin{array}{ll}
			(q-1)(q-3) \cdots (q-(2\ell-3))(q-(\ell-1))										
			& \textup{if $\omega = 1$ and $q \not\in 2\mathbb{Z}$;}\\
			(q-2)(q-4) \cdots (q-2(\ell-2))(q^2 -2(\ell-1)q + \frac{\ell(\ell-1)}{2})		
			& \textup{if $\omega = 1$ and $q \in 2\mathbb{Z}$;}\\
			\end{array}\\
			\!\!\!\begin{array}{ll}
			0
			& \textup{if $\omega = \omega_1$ and $q \not\in 2\mathbb{Z}$;}\\
			(q-2)(q-4) \cdots (q-2(\ell-2))													
			& \textup{if $\omega = \omega_1$ and $q \in 2\mathbb{Z}$;}\\
			\end{array}\\
			\!\!\!\begin{array}{ll}
			0
			& \textup{if $\omega \in \{\omega_{\ell-1},\,\omega_\ell\}$ and $q \not\in 4\mathbb{Z}$;}\\
			2(q-4)(q-8) \cdots (q-2(\ell-3))		
			& \textup{if $\omega \in \{\omega_{\ell-1},\,\omega_\ell\}$ and $q \in 4\mathbb{Z}$}.
			\end{array}
		\end{cases*}
	\end{align}
\end{theorem}
\begin{proof}
	In the case where $\omega = 1$, see \cref{typeDroot}.
	If $\omega =\omega_1$, then $\chi_{\PHI,q}(\omega)$ is equal to the second formula of \cref{typeCroot}.
	When $\omega \in \{\omega_{\ell-1},\omega_\ell\}$, use \cref{dilated} for the second formula of \cref{typeCroot}.
\end{proof}

\subsection{type $E_6$}\ 

\noindent
Let $\PHI$ be a root system of type $E_6$, and $\omega = \omega_j \in W_{A_\circ} \setminus \{1\}$.
Then $o(\omega) = 3$ and $\PHI_\re^\omega$ is of type $G_2$ and 
\begin{align}
	P(\beta^\omega) = \begin{cases*}
		\{0,\pm1,\pm2\}		& if $\beta^\omega$ is a short root of $\PHI_\re^\omega$;\\
		\{0\}			& if $\beta^\omega$ is a long root of $\PHI_\re^\omega$.
	\end{cases*}
\end{align}
Therefore, by \cref{lem6.1}, the equation \cref{chiphiq} can be expressed as follows:
\begin{align}
	\chi_{\PHI,q}(\omega) &= c(q) \cdot \#\Bigset{\pi_T(x) \in T_M[q]}{\frac{p}{3} + (\beta^\omega,x) \not\in \mathbb{Z} \tforall \beta^\omega \in \PHI_\re^\omega,\ p \in P(\beta^\omega)}\\
	&= c(q) \cdot \#\Bigset{\pi_T(x) \in T_M[q]}{
		\begin{lgathered}
			(3\beta^\omega,x) \not\in \mathbb{Z} \text{ if $\beta^\omega \in \PHI_\re^\omega$ is a short root},\\
			(\beta^\omega,x) \not\in \mathbb{Z} \text{ if $\beta^\omega \in \PHI_\re^\omega$ is a long root}
	\end{lgathered}}
\end{align}
Let $\PHI' \ceq (\PHI_\re^\omega)^\vee$.
Then we have
\begin{align}
	\chi_{\PHI,q}(\omega) 
	&= c(q) \cdot \#\Bigset{\pi_T(x) \in T_M[q]}{
		\begin{lgathered}
			(3\beta^\omega,x) \not\in \mathbb{Z} \text{ if $\beta^\omega \in \PHI_\re^\omega$ is a short root},\\
			(\beta^\omega,x) \not\in \mathbb{Z} \text{ if $\beta^\omega \in \PHI_\re^\omega$ is a long root}
	\end{lgathered}}\\
	&= c(q) \cdot \#\bigset{\pi_T(x) \in T_M[q]}{(\beta,x) \not\in \mathbb{Z} \tforall \beta \in \PHI'}\\
	&= c(q) \cdot \chi_{\PHI',M}^\quasi(q).
\end{align}

\begin{theorem}[type $E_6$]
	Let $\PHI$ be a root system of type $E_6$, and $\omega = \omega_j \in W_{A_\circ}$.
	Then 
	\begin{align}
		\chi_{\PHI,q}(\omega) = \begin{cases*}
			\!\!\!\begin{array}{ll}
			(q-1)(q-4)(q-5)(q-7)(q-8)(q-11)				& \textup{if $\omega = 1$ and $\gcd\{6,q\} = 1$;}\\
			(q-2)(q-4)(q-8)(q-10)(q^2-12q+26)			& \textup{if $\omega = 1$ and $\gcd\{6,q\} = 2$;}\\
			(q-3)(q-9)(q^4-24q^3+195q^2-612q+480)		& \textup{if $\omega = 1$ and $\gcd\{6,q\} = 3$;}\\
			(q-6)^2(q^4-24q^3+186q^2-504q+480)			& \textup{if $\omega = 1$ and $\gcd\{6,q\} = 6$;}\\
			\end{array}\\
			\!\!\!\begin{array}{ll}
			0											& \textup{if $\omega \in \{\omega_1,\omega_6\}$ and $\gcd\{6,q\} = 1$;}\\
			0											& \textup{if $\omega \in \{\omega_1,\omega_6\}$ and $\gcd\{6,q\} = 2$;}\\
			2(q-3)(q-9)									& \textup{if $\omega \in \{\omega_1,\omega_6\}$ and $\gcd\{6,q\} = 3$;}\\
			2(q-6)^2									& \textup{if $\omega \in \{\omega_1,\omega_6\}$ and $\gcd\{6,q\} = 6$}.
			\end{array}
		\end{cases*}
	\end{align}
\end{theorem}
\begin{proof}
	In the case where $\omega = 1$, see \cref{typeE6root}.
	When $\omega \in \{\omega_1,\omega_6\}$, use the second formula of \cref{typeG2root}.
\end{proof}

\subsection{type $E_7$}\ 

\noindent
Let $\PHI$ be a root system of type $E_7$, and $\omega = \omega_j \in W_{A_\circ} \setminus \{1\}$.
Then $o(\omega) = 2$ and $\PHI_\re^\omega$ is of type $F_4$ and 
\begin{align}
	P(\beta^\omega) = \begin{cases*}
		\{0,\pm1\}		& if $\beta^\omega$ is a short root of $\PHI_\re^\omega$;\\
		\{0\}			& if $\beta^\omega$ is a long root of $\PHI_\re^\omega$.
	\end{cases*}
\end{align}
Therefore, by \cref{lem6.1}, the equation \cref{chiphiq} can be expressed as follows:
\begin{align}
	\chi_{\PHI,q}(\omega) &= c(q) \cdot \#\Bigset{\pi_T(x) \in T_M[q]}{\frac{p}{2} + (\beta^\omega,x) \not\in \mathbb{Z} \tforall \beta^\omega \in \PHI_\re^\omega,\ p \in P(\beta^\omega)}\\
	&= c(q) \cdot \#\Bigset{\pi_T(x) \in T_M[q]}{
		\begin{lgathered}
			(2\beta^\omega,x) \not\in \mathbb{Z} \text{ if $\beta^\omega \in \PHI_\re^\omega$ is a short root},\\
			(\beta^\omega,x) \not\in \mathbb{Z} \text{ if $\beta^\omega \in \PHI_\re^\omega$ is a long root}
	\end{lgathered}}
\end{align}
Let $\PHI' \ceq (\PHI_\re^\omega)^\vee$.
Then we have
\begin{align}
	\chi_{\PHI,q}(\omega) 
	&= c(q) \cdot \#\Bigset{\pi_T(x) \in T_M[q]}{
		\begin{lgathered}
			(2\beta^\omega,x) \not\in \mathbb{Z} \text{ if $\beta^\omega \in \PHI_\re^\omega$ is a short root},\\
			(\beta^\omega,x) \not\in \mathbb{Z} \text{ if $\beta^\omega \in \PHI_\re^\omega$ is a long root}
	\end{lgathered}}\\
	&= c(q) \cdot \#\bigset{\pi_T(x) \in T_M[q]}{(\beta,x) \not\in \mathbb{Z} \tforall \beta \in \PHI'}\\
	&= c(q) \cdot \chi_{\PHI',M}^\quasi(q).
\end{align}

\begin{theorem}[type $E_7$]
	Let $\PHI$ be a root system of type $E_7$, and $\omega = \omega_j \in W_{A_\circ}$.
	Then 
\begin{align}
	\chi_{\PHI,q}(w) = \begin{cases*}
		\!\!\!\begin{array}{ll}
		(q-1)(q-5)(q-7)(q-9)(q-11)(q-13)(q-17)						& \textup{if $\omega = 1$ and $\gcd\{12,q\} = 1$;}\\
		(q-2)(q-10)(q-13)(q-14)(q^3-24q^2+155q-342)					& \textup{if $\omega = 1$ and $\gcd\{12,q\} = 2$;}\\
		(q-3)(q-9)(q-15)(q^4-36q^3+438q^2-2052q+2289)				& \textup{if $\omega = 1$ and $\gcd\{12,q\} = 3$;}\\
		(q-4)(q-5)(q-8)(q-16)(q^3-30q^2+263q-504)					& \textup{if $\omega = 1$ and $\gcd\{12,q\} = 4$;}\\
		(q-6)(q^6-57q^5+1275q^4-14085q^3+79374q^2-213228q+234360)	& \textup{if $\omega = 1$ and $\gcd\{12,q\} = 6$;}\\
		(q-12)(q^6-51q^5+1005q^4-9675q^3+47784q^2-116064q+120960)	& \textup{if $\omega = 1$ and $\gcd\{12,q\} = 12$;}\\
		\end{array}\\
		\!\!\!\begin{array}{ll}
		0							& \textup{if $\omega = \omega_7$ and $\gcd\{12,q\} = 1$;}\\
		(q-2)(q-10)^2(q-14)			& \textup{if $\omega = \omega_7$ and $\gcd\{12,q\} = 2$;}\\
		0							& \textup{if $\omega = \omega_7$ and $\gcd\{12,q\} = 3$;}\\
		(q-4)(q-8)^2(q-16)			& \textup{if $\omega = \omega_7$ and $\gcd\{12,q\} = 4$;}\\
		(q-6)(q^3-30q^2+268q-552)	& \textup{if $\omega = \omega_7$ and $\gcd\{12,q\} = 6$;}\\
		(q-12)(q^3-24q^2+160q-384)	& \textup{if $\omega = \omega_7$ and $\gcd\{12,q\} = 12$}.
		\end{array}
	\end{cases*}
\end{align}
\end{theorem}
\begin{proof}
	In the case where $\omega = 1$, see \cref{typeE7root}.
	When $\omega = \omega_7$, use the second formula of \cref{typeF4root}.
\end{proof}

\section*{Acknowledgement}

The author would like to thank Professor Masahiko Yoshinaga for the helpful discussions and comments on this research. 
The author also acknowledge support by JSPS KAKENHI, Grant Number 25KJ1735.


\bibliographystyle{amsplain}
\bibliography{bibfile}

\end{document}